\documentclass[a4,10pt]{article}
\author{Mar\'{\i}a de la Paz Tirado Hern\'andez}
\usepackage{amssymb}
\usepackage{amsfonts}
\usepackage{latexsym}

\usepackage{amsmath}
\usepackage[latin1]{inputenc}

\usepackage[all]{xy}

\usepackage{theorem}
\usepackage{graphicx}
\usepackage{pb-diagram}
\usepackage{verbatim}

\usepackage{anysize}
\usepackage{cancel}
\usepackage{color}
\usepackage[colorlinks=true,linkcolor=blue,urlcolor=black]{hyperref}

\usepackage{tikz-cd}

\newtheorem{teo}{Theorem}[section]
\newtheorem{prop}[teo]{Proposition}
\newtheorem{cor}[teo]{Corollary}
\newtheorem{lem}[teo]{Lemma}
\newtheorem{Def}[teo]{Definition}
\newtheorem{nota}[teo]{Remark}

\newtheorem{ej}[teo]{Example}
\newtheorem{hip}[teo]{Hypothesis}
\newtheorem{notacion}[teo]{Notation}
\newtheorem{Contraejemplo}[teo]{Counterexample}

\newenvironment{proof}{
\noindent {\bf  Proof.}\ }{\hfill $\square$\vspace{1ex}\par}

\newenvironment{La induccion}{
\noindent {\it  Assumption.}\ }{\hfill $\diamond$\vspace{1ex}\par}

\long\def\inhibe#1\endinhibe{\relax}

\newcommand{\h}{\langle h \rangle}
\newcommand{\I}{\mathcal I}
\newcommand{\J}{\mathcal J}

\newcommand{\N}{\mathbb N}
\newcommand{\Q}{\mathbb Q}

\newcommand{\PR}{\mathcal P}

\newcommand{\wdelta}{\widetilde{\delta}}

\DeclareMathOperator{\charr}{\rm char}

\DeclareMathOperator{\Hom}{\rm Hom}

\DeclareMathOperator{\Leap}{\rm Leaps}

\DeclareMathOperator{\End}{\rm End}

\DeclareMathOperator{\Aut}{\rm Aut}

\DeclareMathOperator{\ord}{\rm ord}

\DeclareMathOperator{\IDer}{\rm IDer}

\DeclareMathOperator{\Der}{\rm Der}

\DeclareMathOperator{\HS}{\rm HS}

\DeclareMathOperator{\II}{\rm \mathbb I}

\DeclareMathOperator{\Id}{\rm Id}

\usepackage{fancyhdr}
\marginsize{2cm}{2cm}{2.5cm}{2.5cm}

\title{On the behavior of modules of $m$-integrable derivations in the sense of Hasse-Schmidt under base change}
\author{Mar\'{\i}a de la Paz Tirado Hern\'andez\thanks{Partially supported by MTM2016-75027, P12-FQM-2696 and FEDER.} \thanks{Departamento de \'Algebra e Instituto de Matem\'aticas (IMUS), Universidad de Sevilla, Espa\~na.}}

\begin{document}
\date {}
\maketitle
\begin{abstract}
We study the behavior of modules of $m$-integrable derivations of a
commutative finitely generated algebra in the sense of Hasse-Schmidt
under base change. We focus on the case of separable ring extensions over a field
of positive characteristic and on the case where the extension is a
polynomial ring in an arbitrary number of variables.

\smallskip
Keywords: Hasse-Schmidt derivation, Integrability, Base change,
Separable algebras.

MSC 2010: 13N15.
\end{abstract}
\vspace{-0.35cm}
\begin{center} INTRODUCTION \end{center}

Let $k$ be a commutative ring and $A$ a commutative $k$-algebra. A
Hasse-Schmidt derivation of $A$ over $k$ of length $m\in \mathbb N$
or $m=\infty$ is a sequence $D=(D_n)_{n=0}^m$ such that: $$
\begin{array}{ccc}
D_0=\Id_A,&\displaystyle D_n(xy)=\sum_{a+b=n} D_a(x)D_b(y)
\end{array}
$$
for all $x,y\in A$ and for all $n$. For $n\geq 1$, the component $D_n$ of a
Hasse-Schmidt derivation is a differential operator of order $\leq
n$ vanishing at 1, in particular $D_1$ is a $k$-derivation.
Hasse-Schmidt derivations of length $m$, also called higher
derivations of order $m$ (see \cite{Ma}), were introduced by H.
Hasse and F.K. Schmidt (\cite{H-S}) and they have been used by
several authors in different contexts (see \cite{Se}, \cite{Tr} or
\cite{Vo}).

An important notion related with Hasse-Schmidt derivations is {\em
integrability}. Let $m\in \mathbb N$ or $m=\infty$, then we say that
$\delta\in \Der_k(A)$ is $m$-integrable if there exists a
Hasse-Schmidt derivation $D$ of length $m$ such that $\delta=D_1$.
The set of all $m$-integrable $k$-derivations is an $A$-submodule of
$\Der_k(A)$ for all $m$, which is denoted by $\IDer_k(A;m)$. If $k$
has characteristic 0 or $A$ is $0$-smooth over $k$, then any
$k$-derivation is $\infty$-integrable (\cite{Ma}), i.e.
$\Der_k(A)=\IDer_k(A;\infty)$. However, if we consider $k$ a ring of
positive characteristic, then we do not have the same property in
general. Nonetheless, the modules $\IDer_k(A;m)$ have better
properties in some way than $\Der_k(A)$ (see \cite{Mo},
\cite{Fe-Na}) and so their exploration could be interesting for
understanding singularities in positive characteristic.

\smallskip

In this paper we study the behavior of modules of $m$-integrable
$k$-derivations under base change. Namely, if $k\to L$ is a ring
extension and $A$ is a $k$-algebra, the well-known base change map
$L\otimes_k \Der_k(A)\to \Der_k(L\otimes_k A)$ induces, for each
$m\geq 1$ or $m=\infty$, a base change map  $\Phi^{L,A}_m:L\otimes_k
\IDer_k(A;m)\to \IDer_L(L\otimes_k A;m)$. We prove that if $A$ is
finitely generated and $L$ is a polynomial ring over $k$ in an
arbitrary number of variables or a separable $k$-algebra over a
field $k$ of positive characteristic then $\Phi^{L,A}_m$ is an
isomorphism for all $m\geq 1$.

\smallskip

This paper is organized as follows:
In section 1 we recall the definition of Hasse-Schmidt derivations
and give some basic results. In section 2 we prove that an
$I$-logarithmic Hasse-Schmidt derivation of a polynomial ring
$R=k[x_1,\ldots,x_d]$ over a ring $k$ of positive characteristic
(where $I\subseteq R$ is an ideal) can be decomposed into two special
Hasse-Schmidt derivations if its 1-component is zero.

In section 3 we recall some classical results of base change maps
for $k$-derivations and we study the induced maps
$\Phi^{L,A}_m:L\otimes_k \IDer_k(A;m)\to \IDer_L(L\otimes_k A;m)$.
We see that $\Phi^{L,A}_m$ is not surjective in general by giving a
counterexample and we prove that if $L$ is a polynomial ring over
$k$ in an arbitrary number of variables or if $L$ is separable
algebra over a field $k$ of positive characteristic then
$\Phi^{L,A}_m$ is bijective for any finitely generated $k$-algebra
$A$ and for all integers $m$.

\begin{center}
Throughout this paper, all rings (and algebras) are assumed to be
commutative.
\end{center}

\section{Hasse-Schmidt derivations}

In this section, we recall the main definitions of the theory of
Hasse-Schmidt derivations and give some basic results. From now on, $k$ will be a commutative ring and
$A$ a commutative $k$-algebra. We denote $\overline{\mathbb
N}:=\mathbb N \cup \{\infty\}$ and, for each integer $m\geq 1$, we
will write $A[|\mu|]_m:=A[|\mu|]/\langle \mu^{m+1}\rangle$ and
$A[|\mu|]_\infty:=A[|\mu|]$. General references for the definitions
and results in this section are \cite[\S 27]{Ma}, \cite{Na2} and
\cite{Na3}.

\begin{Def}\label{DefHS}
A Hasse-Schmidt derivation (HS-derivation for short) of $A$ (over
$k$)  of length $m\geq 1$ (resp. of length $\infty$) is a sequence
$D:=(D_0,D_1,\ldots, D_m)$ (resp. $D=(D_0,D_1,\ldots)$) of
$k$-linear maps $D_i:A\rightarrow A$, satisfying the conditions:
$$
\begin{array}{ccc}
D_0=\Id_A,&\displaystyle D_n(xy)=\sum_{i+j=n} D_i(x)D_j(y)
\end{array}
$$
for all $x,y\in A$ and for all $n$. We write $\HS_k(A;m)$ (resp.
$\HS_k(A;\infty)=\HS_k(A)$) for the set of HS-derivations of $A$ (over
$k$) of length $m$ (resp. $\infty$).
\end{Def}

For $i\geq 1$, the $D_i$ component of a HS-derivation $D\in \HS_k(A;m)$ is a
$k$-linear differential operator of order $\leq i$ vanishing at $1$.
In particular, $D_1$ is a $k$-derivation.
\medskip

The set $\HS_k(A;m)$ has a natural
group structure with identity $\mathbb{I}=(\Id, 0,\ldots)$ and
$D\circ D'=D''\in \HS_k(A;m)$ with
$
D''_n=\sum_{i+j=n} D_i\circ D_j'
$
for all $n$. We denote by $D^\ast\in \HS_k(A;m)$ the inverse of
$D\in \HS_k(A;m)$. Observe that $D_1^\ast=-D_1$ and that the map:
$(\Id, D_1)\in \HS_k(A;1)\mapsto D_1\in \Der_k(A)$ is a group
isomorphism.
\medskip

Any HS-derivation $D\in \HS_k(A;m)$ is determined by the
$k$-algebra homomorphism
$$
\varphi_D: a\in A \longmapsto
\sum_{i\geq 0}^m D_i(a)\mu^i \in A[|\mu|]_m
$$
satisfying $\varphi_D(a)\equiv a\mod \mu$. If we denote
$$
\Hom_{k-\text{alg}}^\circ (A,A[|\mu|]_m):=\{f\in
\Hom_{k-\text{alg}}(A,A[|\mu|]_m)\ |\ f(a)\equiv a \mod \mu
\mbox{ } \forall a\in A\},
$$
we have a bijection
$$
D\in \HS_k(A;m)\mapsto \varphi_D \in \Hom_{k-\text{alg}}^\circ
\left(A,A[|\mu|]_m\right).
$$
The map $\varphi_D$ can be uniquely
extended to a $k[|\mu|]_m$-algebra automorphism
$\widetilde{\varphi_D}:A[|\mu|]_m\rightarrow A[|\mu|]_m$ with
$\widetilde{\varphi_D}(a)\equiv a_0$ for all $a=\sum_i a_i \mu^i\in A[|\mu|]_m$. If we denote
$$
\Aut_{k[|\mu|]_m-\text{alg}}^\circ (A[|\mu|]_m):=\{f\in
\Aut_{k[|\mu|]_m-\text{alg}}(A[|\mu|]_m)\ |\ f(a)\equiv a_0 \mod \mu
\mbox{ } \forall a\in A[|\mu|]_m\},
$$
we have a group isomorphism
$
D\in \HS_k(A;m) \longmapsto \widetilde{\varphi_D} \in \Aut_{k[|\mu|]_m-\text{alg}}^\circ (A[|\mu|]_m),
$
and for $D,D'\in \HS_k(A;m)$ we have $\varphi_{D\circ D'}:= \widetilde{\varphi_D} \circ
\varphi_{D'}$.
\medskip

A HS-derivation $D$ of $A$ over $k$ of length $m$ can
be understood as a power series
$\sum_{i=1}^m D_i \mu^i$ with coefficients in $\End_k(A)$
and so we can consider $\HS_k(A;m)$ as a subgroup of the group of units of $\End_k(A)[|\mu|]_m$.
\medskip

Additional details for the above material can be found in \cite[\S 5]{Na3}.

\begin{Def}\label{l(D)}
For each HS-derivation $D\in \HS_k(A;m)$ such that $D\neq \mathbb{I}$, we denote $$ \ell(D):=\min\{h\geq 1\mbox{ }|\mbox{ }D_h\neq
0\}
$$
and for $D=\mathbb I$, $\ell(D)=\infty$, i.e. $\ell(D) = \ord (D-\mathbb{I})$.
\end{Def}

The following lemma is clear (see \cite[\S 5]{Na3}).

\begin{lem} \label{l(de la composicion)}
If $D,E\in \HS_k(A;m)$, then $\ell(D\circ E)\geq
\min\{\ell(D),\ell(E)\}$. In particular, if $\ell(D),\ell(E)\geq n$,
then $\ell(D\circ E)\geq n$ and $(D\circ E)_n=D_n+E_n$.
\end{lem}

\subsection{The action of substitution maps}

In this section, we recall some notions and results of \cite[\S 6]{Na3}.

\begin{Def}\label{Def sustitucion}
An $A$-algebra map $\psi:A[|\mu|]_m\rightarrow A[|\mu|]_n$ will be
called a substitution map if $\psi(\mu)\in \langle \mu\rangle$.
We say that a substitution map $\psi:A[|\mu|]_m\rightarrow
A[|\mu|]_n$ has constant coefficients if $\psi(\mu)=\sum_{i\geq 1} a_i\mu^i$
with $a_i\in k$ for all
$i$.
\end{Def}

Compositions of substitution maps (with constant coefficients) are also substitution maps (with constant coefficients).
\medskip

It is clear that for any $f\in \Hom_{k-\text{\rm alg}}^\circ
(A,A[|\mu|]_m)$ and any substitution map $\psi:A[|\mu|]_m \rightarrow
A[|\mu|]_n$, we have
that $\psi \circ f\in \Hom_{k-\text{\rm alg}}^\circ
(A,A[|\mu|]_n)$.

\begin{notacion} Let $\psi:A[|\mu|]_m \rightarrow
A[|\mu|]_n$ be a substitution map and $D\in
\HS_k(A;m)$ a HS-derivation. We denote by $\psi\bullet D\in \HS_k(A;n)$ the
HS-derivation determined by $\varphi_{\psi\bullet D} = \psi\circ \varphi_D$.
In terms of power series, we have:
$$
\psi \bullet D \equiv \psi \bullet \left( \sum_{i=0}^m D_i \mu^i \right) =\sum_{i=0}^m \psi(\mu)^i D_i.
$$
\end{notacion}

\begin{ej}\label{Ejemplos de sustituciones}
In this paper we mainly use three types of substitution maps. Let $D\in \HS_k(A;m)$ be a HS-derivation of length $m\in \overline \N$.
\begin{enumerate}
\item \label{Multporesc} For each $a\in A$, we define $a\bullet D :=\psi \bullet D \in \HS_k(A;m)$ where
$\psi:A[|\mu|]_m \to A[|\mu|]_m$ is given by $\psi(\mu)=a \mu$. Namely:
$a\bullet D=(a^iD_i)_i$.
\item Let $1\leq n\leq m$ with $n\in \overline \N$ and let us consider the projection $\pi_{mn}:
A[|\mu|]_m\to  A[|\mu|]_n$ ($\pi_{mn}(\mu)=\mu$). The {\em truncation} $\tau_{mn}(D)$ is given by $\tau_{mn}(D)=\pi_{mn}\bullet D$, i.e. $\tau_{mn}(D)=(\Id,
D_1,\ldots,D_n)\in \HS_k(A;n)$.
\item For each integer $n\geq 1$, we define $D[n]=\psi\bullet D \in \HS_k(A;nm)$ where $\psi: A[|\mu|]_m\to A[|\mu|]_{nm}$ is the substitution map given by $\psi(\mu)= \mu^n$. Namely:
$$
    D[n]_i=\left\{\begin{array}{ll}
    D_{i/n}& \mbox{ if } i=0\mod n\\
    0&\mbox{ otherwise.}
    \end{array}\right.
    $$
\end{enumerate}
\end{ej}

Substitution maps of type 2. and 3. of Example \ref{Ejemplos de
sustituciones} have constant coefficients. Moreover, if $a\in k$,
the substitution map $a\bullet (-)$ of type 1. has constant coefficients too.
\medskip

The following lemma comes from 8. and Prop. 11  of \cite[\S6]{Na3}.

\begin{lem}\label{composicion y sustituciones}
Let $\phi:A[|\mu|]_m\rightarrow A[|\mu|]_n$ and
$\psi:A[|\mu|]_n\rightarrow A[|\mu|]_s$ be substitution maps and
$D,D'\in \HS_k(A;m)$ HS-derivations. We have the
following properties:
\begin{itemize}
\item[1.] If $\phi$ has constant coefficients, then $\phi\bullet(D\circ D')=(\phi \bullet D)\circ (\phi\bullet D')$.
\item[2.] $\psi\bullet(\phi \bullet D)=(\psi\circ \phi)\bullet D$.
\end{itemize}
\end{lem}

As a straightforward consequence we obtain the following corollary.

\begin{cor}\label{La truncation y el desplazamiento iterado}
Let $D,D^1,\ldots,D^t\in \HS_k(A,m)$ be HS-derivations of
length $m\in \overline{\N}$. The following properties hold:
\begin{itemize}
\item[1.] For each $a\in k$, we have $a\bullet(D^1\circ \cdots \circ  D^t)=(a\bullet D^1)\circ \cdots \circ (a\bullet D^t)$.
\item[2.] $\tau_{mn}(D^1\circ \cdots \circ D^t)=\tau_{mn}(D^1)\circ \cdots \circ \tau_{mn}(D^t)$ for any $n\leq m$.
\item[3.] $(D^1\circ \cdots \circ D^t)[n]=D^1[n]\circ \cdots \circ D^t[n]$ for any $n\geq 1$.
\item[4.] $D[n n']=(D[n])[n']$ for any $n,n'\geq 1$.
\end{itemize}
\end{cor}

The proof of the following lemma is easy and it is left up to the reader (see \cite[\S1.2]{Na2}).

\begin{lem}\label{Relaciones entre operaciones}
Let $D\in \HS_k(A;m)$ be a HS-derivation of length $m\in
\overline \N$, $n\geq 1$ and $q\leq m$. The following properties
hold:
\begin{itemize}
\item[1.] $\left(a^n\bullet D\right)[n]=a\bullet(D[n])$ for all $a\in A$.
\item[2.] $\tau_{mn,m'n}(D[n])=(\tau_{mm'}(D))[n]$ for all $1\leq m'\leq m$.
\item[3.] $\tau_{mq}(a\bullet D)=a\bullet(\tau_{mq}(D))$ for all
$a\in A$.
\end{itemize}
\end{lem}

The following proposition is proved in \cite[Prop. 11]{Na3}.

\begin{prop}\label{Prop 11 de Na3}
Let $\psi:A[|\mu|]_m\rightarrow A[|\mu|]_n$ be a substitution map
with \underline{constant coefficients}. Then, $(\psi \bullet D)^\ast=\psi\bullet
D^\ast$ for each $D\in \HS_k(A;m)$.
\end{prop}

The following two lemmas are clear.

\begin{lem}\label{La ultima componente difiere en una derivacion}
Let $D,E\in \HS_k(A;m)$ be two HS-derivations of length
$m\in \N$ such that $\tau_{m,m-1}(D)=\tau_{m,m-1}(E)$. Then, there
exists $\delta\in \Der_k(A)$ such that $D=E\circ (\Id,\delta)[m]$.
\end{lem}

\begin{lem}\label{Composicion conmuta con las derivaciones en la ultima
componente} Let $D\in \HS_k(A;m)$ be a HS-derivation of
length $m\in \N$ and $\delta\in \Der_k(A)$, then
$D\circ(\Id,\delta)[m]=(\Id,\delta)[m]\circ D$.
\end{lem}

\subsection{Integrable derivations}

In this section, we recall the notion of {\em $n$-integrable derivation}. This notion was introduced in \cite{Ma1} for $n=\infty$. The case of finite $n$ has been studied in \cite{Na2}. We also recall the ``logarithmic point of view'' developed in {\em loc.~cit.}.
From now on, $k$ will be a commutative ring, $A$ a commutative
$k$-algebra and $I\subseteq A$ an ideal.
\medskip

Remember that a $k$-derivation $\delta:A\to A$ is called {\em $I$-logarithmic} if $\delta(I) \subset I$. The set of $I$-logarithmic $k$-derivations is an $A$-submodule of $\Der_k(A)$ and  will be denoted by $\Der_k(\log I)$.
\medskip

If $A$ is a finitely generated $k$-algebra, we may assume that $A$ is the quotient of $R=k[x_1,\dots,x_d]$ by some ideal $I$. There is an exact sequence of $R$-modules:
$$ 0 \longrightarrow I \Der_k(R) \longrightarrow \Der_k(\log I) \longrightarrow \Der_k(A) \longrightarrow 0
$$
where the last map is given by:

\begin{equation}\label{delta barra}
\delta \in \Der_k(\log I) \mapsto \overline{\delta} \in \Der_k(A)\mbox{ with }\overline{\delta}(r+I) = \delta(r) +I \mbox{ for all }r\in R.
\end{equation}

\begin{Def}\label{Log-Int} Let $D\in \HS_k(A;m)$ with $m\in \overline{\mathbb N}$, $I\subset A$ an ideal and $n\geq m$.
\begin{itemize}
\item $D$ is {\em $I$-logarithmic} if $D_i(I)\subseteq I$ for all $i$. The set of $I$-logarithmic HS-derivations is denoted by $\HS_k(\log I;m)$ and $\HS_k(\log I):=\HS_k(\log I;\infty)$. In particular we have
$\Der_k(\log I) \equiv \HS_k(\log I;1)$.
\item More generally, for $r\leq m$, $D$ is {\em $r-I$-logarithmic} if $\tau_{mr}(D)\in \HS_k(\log I;r)$.
\item $D$ is {\em $n$-integrable} if there exists $E\in \HS_k(A,n)$ such that $\tau_{nm}(E)=D$. Any such $E$
will be called an $n$-integral of $D$. If $D$ is $\infty$-integrable
we simply say that $D$ is {\em integrable}. If $m=1$, we write
$\IDer_k(A;n)$ for the set of $n$-integrable derivations and
$\IDer_k(A):=\IDer_k(A;\infty)$.
\item If $D\in \HS_k(\log I;m)$, we say that $D$ is {\em $I$-logarithmically $n$-integrable} if there exists $E\in \HS_k(\log I;n)$ such that $E$ is an $n$-integral of
$D$. We denote $\IDer_k(\log I;n)$ the set of $I$-logarithmically
$n$-integrable derivations  (i.e. for $m=1$) and $\IDer_k(\log
I):=\IDer_k(\log I, \infty)$.
\end{itemize}
\end{Def}

The following lemma is clear.

\begin{lem} Under the above conditions, the following properties hold:
\begin{enumerate}
\item $\HS_k(\log I;m)$ is a subgroup of $\HS_k(A;m)$ for all $m\in \overline \N$.
\item $I$-logarithmicity and $I$-logarithmically $n$-integrability are stable by the action of substitution maps.
\item $\IDer_k(A;n)$ and $\IDer_k(\log I;n)$ are $A$-submodules of $\Der_k(A)$ for all $n\in \overline \N$.
\end{enumerate}
\end{lem}

\begin{Def}\label{leap} Let $s>1$ be an
integer. We say that the $k$-algebra $A$ has a {\em leap} at $s>1$ if the
inclusion $\IDer_k(A;s-1)\supsetneq \IDer_k(A;s)$ is proper. The set
of leaps of $A$ over $k$ is denoted by $\Leap_k(A)$.
\end{Def}

If $k \supset \Q$, it is well-known that $\IDer_k(A;m)=\Der_k(A)$
for all $m\in \overline\N$ and so $A$ has no leaps (see \cite{Ma1}). Let us recall Theorem 27.1 of \cite{Ma}.

\begin{teo}\label{IDer=Der}
If $A$ is $0$-smooth over $k$, then any HS-derivation of length $m<\infty$ over $k$ is $\infty$-integrable.
\end{teo}

Let us recall Theorem 4.1 of \cite{Ti2}.

\begin{teo} \label{Teo Integrabilidad}
Let $k$ be a ring of $\charr(k)=p>0$ (i.e. $\mathbb{F}_p \subset k$) and $A$ a
 $k$-algebra. Then, $\Leap_k(A)\subseteq \{p^\tau\ |\ \tau\geq 1\}$.
\end{teo}

\begin{Def} Let $I\subset A$ be an ideal. An {\em $I$-differential operator} is a ($k$-linear) differential operator $H:A \to A$ such that $H(I)\subset I$.
\end{Def}

In the following lemma we collect some easy results that will be used later. Its proof is left up to the reader.

\begin{lem} \label{lema-resumen}
Let $D\in \HS_k(A;m)$ be a HS-derivation of length $m\in \overline\N$, $n,s\geq 1$ positive integers such that $n\leq m$ and $I\subset A$ an ideal. The following properties hold:
\begin{enumerate}
\item[(a)] 
If $D$ is $(n-1)-I$-logarithmic, then $D[s]\in
\HS_k(A;ms)$ is $(ns-1)-I$-logarithmic.
\item[(b)] 
If $D$ is $n-I$-logarithmic, then $D^\ast$ is
$n-I$-logarithmic too.
\item[(c)] 
If $D$ is $(n-1)-I$-logarithmic, then $D^\ast[s]\in \HS_k(A;ms)$ is $(ns-1)-I$-logarithmic.
\item[(d)] 
Let   $D^1,\ldots, D^t\in \HS_k(A;m)$ be an ordered family of
$(n-1)-I$-logarithmic HS-derivations and let us denote
$D:=D^1\circ \cdots \circ D^t\in \HS_k(A;m)$. Then, $D$ is
$(n-1)-I$-logarithmic and $D_n=\sum_{i=1}^t D_n^i+F_n$ where $F_n$
is an $I$-differential operator of order $\leq
n$.
\item[(e)] 
If $D$ is $(n-1)-I$-logarithmic, then, $D_n^\ast=-D_n+F_n$ where
$F_n$ is an $I$-differential operator of order
$\leq n$.
\item[(f)] 
If $D$ is $(n-1)-I$-logarithmic and $E\in \HS_k(\log I;m)$, then
$D\circ E\in \HS_k(A,m)$ is $(n-1)-I$-logarithmic and $(D\circ
E)_n=D_n+F_n$ where $F_n$ is an $I$-differential operator of order $\leq n$.
\end{enumerate}
\end{lem}

Let us consider a $k$-algebra $A$, $I\subseteq A$ an ideal and $m\in \overline \N$ and we denote $\pi:A\to A/I$ the natural projection. An $I$-logarithmic HS-derivation $D\in \HS_k(\log I;m)$ gives rise to a unique $\overline D\in \HS_k(A/I;m)$ such that $\overline D_r\circ \pi=\pi\circ D_r$ for all $r\geq 0$, and the following diagram is commutative:

$$
\begin{tikzcd}
A \ar[]{r}{\varphi_D} \ar[]{d}{\pi}&
 A[|\mu|]_m\ar[]{d}{\pi_m}\\
A/I \ar[]{r}{\varphi_{\overline D}} & (A/I)[|\mu|]_m
\end{tikzcd}
$$
 where $\pi_m:A[|\mu|]_m\to (A/I)[|\mu|]_m$ is the natural projection. The map  $\Pi_{\HS,m}^I:D\in \HS_k(\log I;m)\to
\overline D\in \HS_k(A/I)$ is clearly a homomorphism of groups.

\begin{lem}\label{Susti y Poli}
Let $I\subseteq A$ an ideal and $\psi:A[|\mu|]_m\to A[|\mu|]_n$ a
substitution map. Let us denote $B=A/I$ and $\psi_B:B[|\mu|]_m\to
B[|\mu|]_n$ the substitution map induced by $\psi$. Then, for each
$D\in \HS_k(\log I;m)$ we have that
$$
\psi_B \bullet \left(\Pi_{\HS,m}^I(D)\right)=\Pi_{\HS,n}^I\left(\psi
\bullet D\right).
$$
\end{lem}

\section{Hasse-Schmidt derivations on polynomial rings}

In this section $k$ will be an arbitrary commutative ring and $R=k[x_1,\ldots,x_d]$ a polynomial ring.

The following result is a straightforward consequence of Theorem \ref{IDer=Der}.

\begin{prop}\label{HSenpolise extiende}
Any HS-derivation of $R$ (over $k$) of length $m\geq 1$
is integrable.
\end{prop}

\begin{prop}{\rm \cite[\bf{Prop. 1.3.4}]{Na2}}\label{HS log es sobreyectivo}
If $I\subseteq R$ is an ideal, the map
$\Pi_{\HS,m}^I:\HS_k(\log I;m)\to \HS_k(R/I;m)$ is a surjective group homomorphism for all $m\in
\overline\N$.
\end{prop}

Let $I\subset R$ be an ideal, $A=R/I$, $m\in \overline\N$ and let us denote by $\Pi^I_m: \IDer_k(\log I;m)\rightarrow
\IDer_k(A;m)$ the map defined as:
\begin{equation} \label{eq:1}
\delta \in \IDer_k(\log I;m) \longmapsto \Pi_m^I(\delta)=\overline{\delta} \in \IDer_k(A;m)
\end{equation}
where $\overline{\delta}$ has been defined in (\ref{delta barra}).

The following proposition is clear thanks to \cite[Corollary 2.1.9]{Na2}.

\begin{prop} \label{Anillo polinomios, logaritmico, map sobrey}
Under the above conditions, the following short sequence of $R$-modules is exact:
$$
0\longrightarrow I(\Der_k(R))\longrightarrow \IDer_k(\log
I;m)\xrightarrow{\Pi_m^I} \IDer_k(R/I;m)\longrightarrow 0.
$$
\end{prop}

\begin{cor}\label{Saltos del cociente}
Under the above conditions, $A$ has a leap at $s> 1$ if and only if
the inclusion $\IDer_k(\log I;s-1) \supsetneq  \IDer_k(\log I;s) $
is proper.
\end{cor}

\subsection{A decomposition of logarithmic HS-derivations in characteristic $\bf{p>0}$}

In this section, $k$ will be a ring of characteristic $p>0$, i.e. $\mathbb{F}_p \subset k$, $R=k[x_1,\ldots,x_d]$ and $I\subset R$ an ideal.
\medskip

We first recall the following theorem.

\begin{teo}\rm{\cite[{\bf Th. 3.14}]{Ti2}}\label{Extraer una integral de De a partir de D}
Let $e,s\geq 1$ be two integers and $D\in \HS_k(R;ep^s)$ an
$(ep^s-1)-I$-logarithmic HS-derivation with $\ell(D)\geq
e$. Then, there exist an integral $D'\in \HS_k(R;p^s)$ of $D_e$
and an $I$-differential operator $H$ of order $\leq p^s$
such that $D'$ is $(p^s-1)-I$-logarithmic and $D_{p^s}'=D_{ep^s}+H$.
\end{teo}

\begin{notacion}\label{Definicion de Cpies}
Let $i,l\geq 1$ be two positive integers such that $i<p^l$. We
define
$
C_{i,l}^p:=\{j\in \mathbb N|\mbox{ } ip^j<p^l\}.
$
\end{notacion}

The proof of the following lemma is clear.

\begin{lem}\label{Cp y potencia de p}
Let $i,l\geq 1$ be two positive integers such that $i<p^l$ and $i$ is not
a power of $p$ and denote $s=\max C_{i,l}^p$. Then $ip^{s+1}>p^l$.
\end{lem}

\begin{notacion}
Let $l\geq 1$ be an integer and $D\in \HS_k(R;p^l)$. We define:
$$\J(l,D):=\{j\in\N\ |\ \ell(D)\leq j\leq p^l, \mbox{ }p\nmid j\}.
$$
Note that if $E\in \HS_k(R;p^l)$ such that $\ell(E)\leq \ell(D)$,
then $\J(l,D)\subseteq \J(l,E)$ and $\J(l,E)\setminus
\J(l,D)=\{j\in\N\ |\ \ell(E)\leq j< \ell(D),\mbox{ } p\nmid j\}$.
For each family $F^j\in \HS_k(R;m)$, $j\in \J(l,D)$, we will write:
$$
\circ_{j\in \J(l,D)} F^j=F^{p^l-1}\circ \cdots \circ F^{\ell(D)}
$$
(observe that we have chosen the decreasing ordering),
where $F^j=\mathbb{I}$ if $j\not\in\J(l,D)$.
\end{notacion}

\begin{prop}\label{Descomposicion detallada de HS en parte no log y partes log}
Let $l>0$ be an integer and let us denote $s_j:=\max C_{j,l}^p$ for each integer $j$ with $1\leq j \leq p^l$. Then, for any $\left(p^l-1\right)-I$-logarithmic HS-derivation
$D\in \HS_k(R;p^l)$ with $\ell(D)>1$, there exist:
\begin{itemize}
\item a $(p^{l-1}-1)-I$-logarithmic
HS-derivation $T\in \HS_k(R;p^{l-1})$,
\item a $(p^{s_j+1}-1)-I$-logarithmic HS-derivation $F^j\in
\HS_k(R;p^{s_j+1})$, for each $j\in \J(l,D)$, and
\item an $I$-differential operator $H$ of order $\leq p^l$
\end{itemize}
such that $T_{p^{l-1}}=D_{p^l}+H$ and
$$
D=T[p]\circ \left(\circ_{j\in \J(l,D)} \left(\psi^j\bullet
F^j\right)\right),
$$
where $\psi^j:R[|\mu|]_{p^{s_j+1}} \rightarrow R[|\mu|]_{p^l}$ is the substitution map given by $\psi^j(\mu)= \mu^j$.
\end{prop}

\begin{proof}  First, note that $\psi^j$ is well-defined for
all $j\in \J(l,D)$ because $jp^{s_j+1}\geq p^l$ by definition of
$s_j$. Moreover, observe that $\psi^j\bullet
E=\tau_{jp^{s_j+1},p^l}(E[j])$ for any $E\in \HS_k(R;p^{s_j+1})$. If
$\ell(D)=\infty$, then $D=\mathbb{I}$, $\J(l,D)=\emptyset$ and we may take
$T=\mathbb{I}$ to obtain the result. Let us suppose that $\ell(D)$ is finite, i.e.
$1<\ell(D)\leq p^l$. We proceed by decreasing induction
on $\ell(D)$.
\smallskip

Assume that $\ell(D)=p^l$. Then, $\J(l,D)=\emptyset$ and
$D=(\Id,\delta)[p^l]=(\Id,\delta)\left[p^{l-1}\right][p]$ (see
Corollary \ref{La truncation y el desplazamiento iterado}). So, if
we put $T:=(\Id,\delta)\left[p^{l-1}\right]$, we have the result.
Let us suppose that the proposition is true for all HS-derivations such that $\ell(\ast)>i$ and let us take
a $(p^l-1)-I$-logarithmic HS-derivation $D\in \HS_k(R;p^l)$ with
$1<\ell(D)=i<p^l$. We divide the proof in two cases:
\begin{itemize}
\item[1.]{\it  If $i$ is a power of $p$.}
\end{itemize}
Let us write $i=p^t$ where $t<l$. Since $\ell(D)>1$, then $t\geq 1$
and we can see $D\in \HS_k(R;p^tp^{l-t})$. By Theorem \ref{Extraer
una integral de De a partir de D}, there exist an integral $F\in
\HS_k(R;p^{l-t})$ of $D_{p^t}$ and an $I$-differential operator $H$
of order $\leq p^{l-t}$ such that $F$ is $(p^{l-t}-1)-I$-logarithmic
and $F_{p^{l-t}}=D_{p^l}+H$. Then, by Lemma \ref{lema-resumen}, (c),
and Proposition \ref{Prop 11 de Na3}, $F^\ast[p^t]=(F[p^t])^\ast\in
\HS_k(R;p^l)$ is $(p^l-1)-I$-logarithmic. Moreover,
$(F[p^t])^\ast_{p^t}=F^\ast_1=-D_{p^t}$ and, by Lemma
\ref{lema-resumen}, (e),
$(F[p^t])^\ast_{p^l}=F^\ast_{p^{l-t}}=-D_{p^l}-H+E$ where $E$ is an
$I$-differential operator of order $\leq p^{l-t}$. We define
$D':=(F[p^t])^\ast\circ D$. By Lemma \ref{l(de la composicion)},
$\ell(D')>i=p^t$ and, by Lemma \ref{lema-resumen}, (d), $D'$ is
$(p^l-1)-I$-logarithmic and $D'_{p^{l}}=D_{p^l}-D_{p^l}+H'=H'$ where
$H'$ is an $I$-differential operator of order $\leq p^l$. So, $D'\in
\HS_k(\log I;p^l)$. We apply the induction hypothesis to $D'$ and we
obtain that
$$
D'=T'[p]\circ \left(\circ_{j\in \J(l,D')} \left( \psi^j\bullet
F^j\right)\right)
$$
where $F^j\in \HS_k(R;p^{s_j+1})$ is $(p^{s_j+1})-I$-logarithmic
for $j\in\J(l,D')$ and $T'\in \HS_k\left(R;p^{l-1}\right)$ is
$\left(p^{l-1}-1\right)-I$-logarithmic with
$T'_{p^{l-1}}=D'_{p^l}+\mbox{some $I$-diff. op. of order} \leq p^l$. Since $D'\in
\HS_k(\log I;p^l)$, we have that $T'\in \HS_k(\log I;p^{l-1})$. We
put $F^j=\II\in \HS_k(\log I;p^{s_j+1})$ for all $j\in
\J(l,D)\setminus \J(l,D')$.  By Corollary \ref{La truncation y el
desplazamiento iterado},
$$
D=F[p^t]\circ T'[p]\circ \left(\circ_{j\in
\J(l,D)}\left(\psi^j\bullet
F^j\right)\right)=\left(F\left[p^{t-1}\right]\circ T'\right)[p]\circ
\left(\circ_{j\in \J(l,D)}\left(\psi^j\bullet F^j\right)\right).
$$
By Lemma \ref{lema-resumen}, (a), $F[p^{t-1}]$ is
$(p^{l-1}-1)-I$-logarithmic. Moreover,
$F[p^{t-1}]_{p^{l-1}}=F_{p^{l-t}}=D_{p^l}+H$. So, by Lemma
\ref{lema-resumen}, (f),
$T:=F\left[p^{t-1}\right]\circ T'\in \HS_k(R;p^{l-1})$ is
$(p^{l-1}-1)-I$-logarithmic and
$T_{p^{l-1}}=F[p^{t-1}]_{p^{l-1}}+\mbox{ some } I\mbox{-diff. op. of order }\leq p^l=D_{p^l}+\mbox{some $I$-diff. op. of order }\leq p^l$ and we have
the proposition in this case.
\begin{itemize}
\item[2.] {\it If $i$ is not a power of $p$.}
\end{itemize}
Since $i$ is not a power of $p$, by Lemma \ref{Cp y potencia de p},
$ip^{s_i+1}>p^l$ where $s_i=\max C^p_{i,l}$. We consider
$\tau_{p^l,ip^{s_i}}(D)\in \HS_k(\log I;ip^{s_i})$. If $s_i\geq 1$,
then $D_i$ is $I$-logarithmically $p^{s_i}$-integrable by  Theorem
\ref{Extraer una integral de De a partir de D}. If $s_i=0$, then
$D_i\in \Der_k(\log I)$. In both cases, since leaps only occur at
powers of $p$ (Theorem \ref{Teo Integrabilidad} and Corollary
\ref{Saltos del cociente}), we have that $D_i$ is
$I$-logarithmically $(p^{s_i+1}-1)$-integrable. Thanks to
Proposition \ref{HSenpolise extiende}, we can integrate any
$I$-logarithmic $(p^{s_i+1}-1)$-integral of $D_i$ so,  there exists
$F\in \HS_k(R;p^{s_i+1})$ a $(p^{s_i+1}-1)-I$-logarithmic integral
of $D_i$. Then, by Lemma \ref{lema-resumen} (c),
$F^\ast[i]\in\HS_k(R;ip^{s_i+1})$ is $ip^{s_i+1}-I$-logarithmic. By
Proposition \ref{Prop 11 de Na3}, $\psi^i\bullet
F^\ast=\left(\psi^i\bullet F\right)^\ast\in \HS_k(\log I;p^l)$ and
$(\psi^i \bullet F)^\ast_i=F[i]^\ast_i=-D_i$.

\begin{itemize}
\item[a.] If $i\not\equiv 0\mod p$, then by Lemma \ref{lema-resumen} (f),
and Lemma \ref{l(de la composicion)}, $D':=D\circ
(\psi^i\bullet F)^\ast$ is $(p^l-1)-I$-logarithmic with $\ell(D')>i$
where $D'_{p^l}=D_{p^l}+H$ with $H$ an $I$-differential
operator of order $\leq p^l$. We apply the induction hypothesis to $D'$ and we obtain
that
$$
D'=T[p]\circ \left(\circ_{j\in \J(l,D')}\left( \psi^j\bullet
F^j\right)\right) \Rightarrow D=T[p]\circ \left(\circ_{j\in
\J(l,D')}\left( \psi^j\bullet F^j\right)\right) \circ(\psi^i \bullet
F)
$$
where $T\in \HS_k(R;p^{l-1})$ is $(p^{l-1}-1)-I$-logarithmic with
$T_{p^l-1}=D'_{p^l}+\mbox{some }I\mbox{-diff. op. of order }\leq p^l=D_{p^l}+H'$ where
$H'$ is an $I$-differential operator of order $\leq p^l$. Then, we put
$F^i=F\in \HS_k(R;p^{s_i+1})$ and $F^j=\II\in \HS_k(\log
I;p^{s_j+1})$ for $j\in \J(l,D)\setminus (\J(l,D') \cup \{i\})$ and
we have the result.
\item[b.] If $i$ is a multiple of $p$, then by Lemma \ref{lema-resumen} (d), and Lemma \ref{l(de la composicion)},
$D':=\left(\psi^i\bullet F\right)^\ast\circ D$ is
$(p^l-1)-I$-logarithmic with $\ell(D')>i$ and $D'_{p^l}=D_{p^l}+H$
where $H$ is an $I$-differential operator of order $\leq p^l$. Then, we
apply the induction hypothesis to $D'$ and we have that
$$
D=\left(\psi^i\bullet F\right)\circ T'[p] \circ \left(\circ_{j\in
\J(l,D')} \left(\psi^j \bullet F^j \right)\right)
$$
where $T'\in \HS_k(R;p^{l-1})$ is $(p^{l-1}-1)-I$-logarithmic with
$T'_{p^{l-1}}=D'_{p^l}+\mbox{some $I$-diff. op. of order }\leq p^l=D_{p^l}+H'$
where $H'$ is an $I$-differential operator of order $\leq p^l$. We put
$F^j=\II$ for all $j\in \J(l,D)\setminus \J(l,D')$. On the other
hand, by Corollary \ref{La truncation y el desplazamiento iterado}
and Lemma \ref{Relaciones entre operaciones},
$$
\psi^i\bullet
F=\tau_{ip^{s_i+1},p^l}(F[i])=\tau_{ip^{s_i+1},p^l}(F[i/p][p])=
\tau_{ip^{s_i},p^{l-1}}(F[i/p])[p]
$$
Since $F$ is $(p^{s_i+1}-1)-I$-logarithmic, $F[i/p]$ is
$(ip^{s_i}-1)-I$-logarithmic by Lemma \ref{lema-resumen} (a), and,
since $ip^{s_i}>p^{l-1}$, $\tau_{{ip^s_i},p^{l-1}}(F[i/p])\in
\HS_k(\log I,p^{l-1})$. By Lemma \ref{lema-resumen} (f),
$T:=\tau_{ip^{s_i},p^{l-1}}(F[i/p])\circ T'$ is
$(p^{l-1}-1)-I$-logarithmic and $T_{p^{l-1}}=T'_{p^{l-1}}+\mbox{
some $I$-diff. op. of order }\leq p^l=D_{p^l}+H''$ where $H''$ is an $I$-differential
operator of order $\leq p^l$. Since $ D=T[p]\circ \left(\circ_{j\in
\J(l,D)} \left(\psi^j\bullet F^j\right)\right)$, we have the
proposition.
\end{itemize}
\end{proof}

\begin{cor}\label{Descomposicion de HS en partes no log y log}
Let $l\geq 1$ be an integer and $D\in \HS_k(R;p^l)$ a $(p^l-1)-I$-logarithmic HS-derivation with $\ell(D)>1$. Then, there exist $F\in
\HS_k(\log I;p^l)$ with $\ell(F)>1$ and a
$(p^{l-1}-1)-I$-logarithmic HS-derivation $T\in \HS_k(R;p^{l-1})$  such that
$D=T[p]\circ F$.
\end{cor}

\begin{proof} From Proposition \ref{Descomposicion detallada
de HS en parte no log y partes log}, we have that
$$
D=T[p]\circ \left(\circ_{i\in \J(l,D)} \left(\psi^i\bullet
F^i\right)\right)
$$
for some $(p^{l-1}-1)-I$-logarithmic HS-derivation  $T\in
\HS_k(R;p^{l-1})$ and some $(p^{s_i+1}-1)-I$-logarithmic
HS-derivations $F^i\in \HS_k(R;p^{s_i+1})$, for $i\in \J(l,D)$ and
$s_i=\max C^p_{i,l}$. Since $\psi^i\bullet
F^i=\tau_{ip^{s_i+1},p^l}(F^i[i])$ and $F^i[i]$ is
$(ip^{s_i+1}-1)-I$-logarithmic by Lemma \ref{lema-resumen} (a),
then $\psi^i\bullet F^i\in \HS_k(\log I;p^l)$ because $i\not\equiv
0\mod p$ and, by Lemma \ref{Cp y potencia de p}, $ip^{s_i+1}>p^l$.
Hence, $F:=\circ_{i\in \J(l,D)}\left( \psi^i\bullet F^i\right)\in
\HS_k(\log I;p^l)$. Moreover, $\ell(F^i[i])>1$ for all $j\in \J(l,D)$, so
$\ell(\psi^i\bullet F^i)>1$ and $\ell(F)>1$ by Lemma \ref{l(de la
composicion)}.
\end{proof}

\section{Base change}

Let $k$ be a commutative ring, $k \to L$ a ring extension
and $A$ a commutative finitely generated $k$-algebra. We denote
$A_L=L\otimes_k A$. In this section, we study the relationship
between $\IDer_k(A;m)$ and $\IDer_L(A_L;m)$ under suitable hypotheses on the ring extension $k \to L$.

\subsection{Base change for derivations}

For any commutative $k$-algebra $A$, let us denote $A_L:= L\otimes_k A$. For each $k$-derivation $\delta:A \to A$ let us denote by $\wdelta: A_L \to A_L$ the natural $L$-linear extension given by $\wdelta(c \otimes a) = c \otimes \delta(a)$ for all $c\in L$ and all $a\in A$, which is a $L$-derivation. The map
$ \delta \in \Der_k(A) \mapsto \wdelta \in \Der_L(A_L)$, being $A$-linear, gives rise to an $A_L$-linear base change map:
\begin{equation*}
\begin{array}{rccc}
\Phi^{L,A}:&L\otimes_k \Der_k(A) = A_L\otimes_A \Der_k(A)&\longrightarrow&\Der_L(A_L)\\
&c\otimes\delta&\longmapsto&c\wdelta.
\end{array}
\end{equation*}
The above map can be also described through the base change isomorphism for the module of differential forms
$ A_L \otimes_A\Omega_{A/k} = L \otimes_k \Omega_{A/k} \stackrel{\sim}{\to} \Omega_{A_L/L}$, namely:
$$ L\otimes_k \Der_k(A) \simeq L\otimes_k \Hom_A(\Omega_{A/k},A) \longrightarrow \Hom_{A_L}(\Omega_{A/k},A_L) \simeq \Hom_{A_L}(A_L \otimes_A\Omega_{A/k},A_L) \simeq \Der_L(A_L).
$$

If $I\subseteq A$ is an ideal, the map $\Phi^{L,A}:L\otimes_k\Der_k(A)\to \Der_L(A_L)$ induces new $A_L$-linear maps:
$$
\begin{array}{ccc}
\Phi^{L,A}_{\text{\rm ind}}: L\otimes_k (I (\Der_k(A)))\to
I^e\Der_L(A_L)&\mbox{ and } &\Phi^{L,A}_{\text{\rm ind}}: L\otimes_k
\Der_k(\log I)\to \Der_L(\log I^e).
\end{array}
$$

When $A=R=k[x_1,\dots,x_d]$ is a polynomial ring, then $R_L=L[x_1,\dots,x_d]$ is also a polynomial ring and since the module of derivations of a polynomial ring in a finite number of variables is free with basis the partial derivatives, we deduce that the map $\Phi^{L,R}$ is an isomorphism for an arbitrary ring extension $k \to L$.
\medskip

We denote
$I^e=IR_L=IL[x_1,\ldots,x_d]$ the extended ideal of $I$ in $R_L$. It is clear that the following diagram is commutative:
\begin{equation}  \label{eq:diag-base-change-derlog}
\begin{tikzcd}
 &
L \otimes_k \left(I (\Der_k(R))\right) \ar[]{r}{} \ar[]{d}{\Phi^{L,R}_{\text{\rm ind}}}  & L \otimes_k \Der_k(\log I) \ar[]{r}{} \ar[]{d}{\Phi^{L,R}_{\text{\rm ind}}} & L \otimes_k \Der_k(A)  \ar[]{r}{}  \ar[]{d}{\Phi^{L,A}} &0 \\
0 \ar[]{r}{} & I^e \Der_L(R_L)  \ar[]{r}{} &  \Der_L(\log I^e)
\ar[]{r}{} & \Der_L(A_L) \ar[]{r}{} & 0.
\end{tikzcd}
\end{equation}
Moreover, it
has exact rows and the left vertical arrow is surjective, and if $L$ is flat over $k$, then the top row is also left exact, the left vertical arrow is bijective and the middle vertical arrow is injective.
\medskip

\begin{prop} \label{prop:equiv-base-extension-log-der}
Under the above hypotheses, if $k \to L$ is a flat ring extension, then the following properties are equivalent:
\begin{enumerate}
\item[(a)] The map $\Phi^{L,R}_{\text{\rm ind}}: L \otimes_k \Der_k(\log I) \to \Der_L(\log I^e)$ is an isomorphism.
\item[(b)] The map $\Phi^{L,A}: L \otimes_k \Der_k(A) \to \Der_L(A_L)$ is an isomorphism.
\end{enumerate}
Moreover, both properties hold if $I$ is finitely generated (i.e. if $A$ is finitely presented over $k$).
\end{prop}

\begin{proof} The equivalence (a) $\Leftrightarrow$ (b) comes from the {\em five lemma}. The last statement is well known (cf. \cite[Prop. (16.5.11)]{ega_iv_4}) but for the sake of completeness we recall its proof:
from the {\em second fundamental exact sequence}
$$ I/I^2 \to A \otimes_R \Omega_{R/k} \to \Omega_{A/k} \to 0
$$
we deduce that, if $I$ is finitely generated, then $\Omega_{A/k}$ is a finitely presented $A$-module and so
\begin{eqnarray*}
& L \otimes_k \Der_k(A) \simeq A_L \otimes_A \Der_k(A) \simeq A_L \otimes_A \Hom_A(\Omega_{A/k},A)
\simeq \Hom_A(\Omega_{A/k},A_L) \simeq
& \\
& \Hom_{A_L}(A_L \otimes_A \Omega_{A/k},A_L) \simeq  \Hom_{A_L}(\Omega_{A_L/L},A_L) \simeq \Der_L(A_L).
\end{eqnarray*}
\end{proof}

We also have the following result for any ideal $I \subset R=k[x_1,\ldots,x_d]$ and for any finitely generated $k$-algebra $A=R/I$.

\begin{prop} \label{prop:free-extension-der}
Under the above hypotheses, if $k \to L$ is a free ring extension ($L$ is a free $k$-module) and $A = R/I$ is a finitely generated $k$-algebra, then properties (a) and (b) in Proposition \ref{prop:equiv-base-extension-log-der} hold.
\end{prop}

\begin{proof} Since $L$ is a (faithfully) flat extension of $k$, after Proposition \ref{prop:equiv-base-extension-log-der} we only need to prove that the map $\Phi^{L,R}_{\text{\rm ind}}: L \otimes_k \Der_k(\log I) \to \Der_L(\log I^e)$ is surjective. Let $ \mathcal B= \{a_i, i \in \I\}$ be a $k$-basis of $L$ and
$\varepsilon: R_L \to R_L$ an $I^e$-logarithmic derivation. Since $\mathcal B$ is also a $R$-basis of $R_L$, there is a finite subset $\I_0 \subseteq \I$ and unique elements $r_{j i} \in R$, $1\leq j \leq d$, $i \in \I_0$ such that $\varepsilon(x_j) = \sum_{i\in\I_0} r_{j i} a_i$ for all $j=1,\dots,d$. We have
$$ \varepsilon = \sum_{j=1}^d \varepsilon(x_j) \partial_j = \sum_{i\in\I_0} a_i \wdelta_i = \Phi^{L,R} \left(  \sum_{i\in\I_0} a_i \otimes \delta_i \right),\ \text{with}\ \delta_i =  \sum_{j=1}^d r_{j i} \partial_j  \in \Der_k(R),
$$
and for each $f\in I$ we have $\varepsilon(f) = \sum_{i\in\I_0} a_i \delta_i(f) \in I^e$ and so $\delta_i(f)\in I$. We deduce that each $\delta_i$ is $I$-logarithmic and so $\varepsilon$ belongs to the image of $\Phi^{L,R}_{\text{\rm ind}}$.
\end{proof}

\subsection{Base change for integrable derivations}

\begin{prop} \label{Extension de D}
Let $A$ be a $k$-algebra, $I\subset A$ an ideal, $k \to L$ a ring extension, $I^e = I A_L$ the extended ideal and $m\in
\overline \N$. For any HS-derivation $D\in \HS_k(A;m)$, there is a unique HS-derivation $\widetilde{D}\in
\HS_L(A_L;m)$ such that the following diagram is commutative:
$$
\begin{tikzcd}
A \ar[r, "\varphi_D" ]  \ar[d, "\text{\rm nat.}"'] & A[|\mu|]_m \ar[d, "\text{\rm nat.}"]\\
A_L \ar[r, "\varphi_{\widetilde{D}}"] &  A_L[|\mu|]_m.
\end{tikzcd}
$$
Moreover, if $D$ is $I$-logarithmic, then $\widetilde{D}$ is $I^e$-logarithmic.
\end{prop}

Observe that for $m=1$, we know that $\Der_k(R)\equiv \HS_k(R;1)$
and the extension process $D \mapsto \widetilde{D}$ described in
Proposition \ref{Extension de D} coincides with the usual extension
$\delta \mapsto \wdelta$ of derivations.

\begin{lem} \label{Propiedades Extension}
Let $A$ be a $k$-algebra, $I\subset A$ an ideal, $k \to L$ a ring extension, $m\in
\overline \N$, $n\leq m$, $D\in \HS_k(A;m)$ a HS-derivation and $\psi:A[|\mu|]_m\rightarrow  A[|\mu|]_n$ a substitution map. The following properties hold:
\begin{enumerate}
\item[(1)] The map $D\in \HS_k(A;m) \mapsto \widetilde{D} \in \HS_L(A_L;m)$ is a group homomorphism.
\item[(2)] $\widetilde{\psi\bullet D}=\widetilde{\psi}\bullet \widetilde D$, where $\widetilde{\psi}:A_L[|\mu|]_m\rightarrow A_L[|\mu|]_n$ is the substitution map induced by $\psi$.
\item[(3)] If $D$ is $n-I$-logarithmic, then $\widetilde D$ is $n-I^e$-logarithmic.
\end{enumerate}
\end{lem}

\begin{lem}\label{Extension y poli}
Let $I\subseteq A$ be an ideal, $B=A/I$ and $I^e=IA_L$ the extended
ideal. Then, for each $D\in \HS_k(\log I;m)$,
$$
\widetilde{\Pi_{\HS,m}^I(D)}=\Pi_{\HS,m}^{I^e}\left(\widetilde
D\right)
$$
(where $\widetilde{\Pi^I_{\HS,m}(D)}$ is the extension of
$\Pi^I_{\HS,m}(D)\in \HS_k(B;m)$ to $B_L=A_L/I^e$ and $\widetilde
D\in \HS_L(\log I^e;m)\subseteq \HS_L(A_L;m)$).
\end{lem}

\begin{cor}   \label{Integrales_y_extension}
Under the hypotheses of Lemma \ref{Propiedades Extension}, let
$\delta:A \to A$ be a $k$-derivation (resp. an
$I$-logarithmic $k$-derivation). If $\delta$ is
$m$-integrable (resp. $I$-logarithmically $m$-integrable), then
$\wdelta$ is also $m$-integrable (resp. $I^e$-logarithmically
$m$-integrable).
\end{cor}

\begin{proof} Let us suppose that $\delta\in \IDer_k(A;m)$ and let us
consider an $m$-integral $D\in \HS_k(A;m)$ of $\delta$, i.e.
$D_1=\delta$. From Proposition \ref{Extension de D}, $\widetilde D
\in \HS_L(A_L;m)$ is an $m$-integral of $\widetilde D_1=\widetilde
\delta$, i.e. $\widetilde\delta\in \IDer_k(A;m)$. Moreover, if
$\delta\in \IDer_k(\log I;m)$, then we can consider $D\in \HS_k(\log
I;m)$ and, by Proposition \ref{Extension de D}, $\widetilde D\in
\HS_L(\log I^e;m)$. Hence, $\widetilde\delta\in \IDer_L(\log
I^e;m)$.
\end{proof}

As a consequence of the above corollary, base change maps $\Phi^{L,A}: L\otimes_k \Der_k(A) \to \Der_L(A_L)$ and $\Phi^{L,A}_{\text{\rm ind}}: L\otimes_k \Der_k(\log I) \to \Der_L(\log I^e)$
induce, for each $m\in\overline{\N}$, new $A_L$-linear base change maps:
$$
 \Phi_m^{L,A}: L\otimes_k \IDer_k(A;m) \longrightarrow \IDer_L(A_L;m),\quad
 \Phi^{L,A}_{\text{\rm ind},m}: L\otimes_k \IDer_k(\log I; m) \to \IDer_L(\log I^e; m).
$$

From now on, we assume that $L$ is flat over $k$ and $A$ a finitely generated
$k$-algebra. Then, we can put  $A=R/I$ where
$R=k[x_1,\ldots,x_d]$ is a polynomial ring and $I\subset R$ an ideal.

From the exact sequence in Proposition \ref{Anillo polinomios, logaritmico, map sobrey}, we obtain for each $m\in\overline{\N}$ a commutative diagram with exact rows (compare with (\ref{eq:diag-base-change-derlog})):
\begin{equation}  \label{eq:diag-base-change-Iderlog}
\begin{tikzcd}
0 \ar[]{r}{} &
L \otimes_k \left(I(\Der_k(R))\right) \ar[]{r}{} \ar[]{d}{\Phi^{L,R}_{\text{\rm ind}}}  & L \otimes_k \IDer_k(\log I;m) \ar[]{rr}{\Id\otimes\Pi^I_m} \ar[]{d}{\Phi^{L,R}_{\text{\rm ind},m}} && L \otimes_k \IDer_k(A;m)  \ar[]{r}{}  \ar[]{d}{\Phi^{L,A}_m} &0 \\
0 \ar[]{r}{} &
I^e \Der_L(R_L)  \ar[]{r}{} &  \IDer_L(\log I^e;m) \ar[]{rr}{\Pi^{I^e}_m} && \IDer_L(A_L;m) \ar[]{r}{} & 0.
\end{tikzcd}
\end{equation}
Moreover the left vertical arrow is bijective and the middle vertical arrow is injective.

The proof of the following lemma is clear.

\begin{lem}\label{Inyectividad/Sobreyectividad}
Under the above hypotheses, the following properties hold:
\begin{itemize}
\item[1.] $\Phi^{L,A}_m$ is injective.
\item[2.] $\Phi^{L,R}_{\text{\rm ind},m}$ is surjective if and only if $\Phi^{L,A}_m$ is surjective.
\end{itemize}
\end{lem}

 Moreover, we have the following result about leaps.

\begin{lem}\label{Saltos y cambio de base}
Assume that $L$ is faithfully flat over $k$ and $A$ a finitely
generated $k$-algebra. If $\Phi^{L,A}_m$ is surjective for all
$m\geq 1$ then,
$$
\Leap_k(A)=\Leap_L(A_L).
$$
\end{lem}

\begin{proof}
Since $L$ is flat over $k$, we have that $\Phi^{L,A}_m$ is bijective
so, $\IDer_L(A_L;m)=\IDer_L(A_L;m-1)$ if and only if
$$
\IDer_L(A_L;m-1)/\IDer_L(A_L;m)=0 \Leftrightarrow L\otimes \left(\IDer_k(A;m-1)/\IDer_k(A;m)\right)=0
$$
Since $L$ is faithfully flat over $k$, the last equality holds if
and only if $\IDer_k(A;m-1)/\IDer_k(A;m)=0$ and we have the result.

\end{proof}

In the rest of this section, we will study the surjectivity of
$\Phi_m^{L,A}$. Let us start by giving a counterexample.

\begin{Contraejemplo}\label{Contraejemplo}
Let us consider $k=\mathbb F_2(s,t)$ the quotient field of $\mathbb
F_2[s,t]$ and $L=\overline k$ the perfect closure of $k$. We denote
$A:=k[x,y]/\langle h\rangle$ where $h\in k[x,y]$ is the irreducible
polynomial $x^2+y^2+tx^4+sy^4$. Then, $\Phi^{L,A}_4$ is not
surjective.
\end{Contraejemplo}

To prove this counterexample, we need to calculate the 4-integrable
derivations of $A$ (resp. $A_L$) over $k$ (resp. over $L$). To do
this, we use two general results:

\begin{prop}{\rm \cite[{\bf Prop. 2.10}]{Ti1}}\label{I^p -logarithmicas}
Let $k$ be a unique factorization domain of characteristic $p>0$,
$R=k[x_1,\ldots,x_d]$ the polynomial ring over $k$ and $h$ a
polynomial of $R$. For all $n\in \mathbb N$, we have:
$$
\IDer_k(\log h;n)=\IDer_k(\log h^p;np).
$$
\end{prop}

\begin{prop}{\rm\cite[{\bf Prop. 2.2.4}]{Na2}}\label{Narvaez Prop 2.2.4}
Let $h\in R=k[x_1,\ldots,x_d]$, $I=\h$ and $J^0=\langle
\partial_1(h),\ldots,\partial_d(h)\rangle$ the gradient ideal. If
$\delta:R\rightarrow R$ is an $I$-logarithmic $k$-derivation with
$\delta\in J^0\Der_k(R)$, then $\delta$ admits an $I$-logarithmic
integral $D\in \HS_k(\log I)$ with $D_i(h)=0$ for all $i>i$. In
particular, if $\delta(h)=0$, the integral $D$ can be taken with
$\varphi_D(h)=h$.
\end{prop}

\noindent{\bf Proof of Counterexample \ref{Contraejemplo}.} To
calculate the $m$-integrable derivations of $A$, we will follow the
same step of Example 7 of \cite{Ma1}. Let us suppose that there
exists $\delta\in \IDer_k(A;m)$ and $D\in \HS_k(A;m)$ an integral of
$\delta$. Let us consider
$$
\begin{array}{rccl}
\varphi_D:&A&\rightarrow& A[|\mu|]\\
&x&\mapsto&x+u_1\mu+u_2\mu^2+\cdots\\
&y&\mapsto&y+v_1\mu+v_2\mu^2+\cdots
\end{array}
$$
where $u_i,v_i\in A$. To $\varphi_D$ be well defined, we need
$\varphi_D(h)=0$, i.e.
$$
(x+u_1\mu+u_2\mu^2+\cdots)^2+(y+v_1\mu+v_2\mu^2+\cdots)^2+
t(x+u_1\mu+u_2\mu^2+\cdots)^4+s(y+v_1\mu+v_2\mu^2+\cdots)^4=0.
$$
By looking at the coefficient of $\mu^2$ in the previous equation, we deduce
$u_1^2+v_1^2=(u_1+v_1)^2=0$, and since $A$ is a domain, $u_1=v_1$.
By looking at the  coefficient of $\mu^4$, we deduce
$u_2^2+v_2^2+tu^4_1+sv^4_1=0$. We can write $w=u_2+v_2$ and
$u=u_1=v_1$, and we obtain the equation:
$$
w^2+(t+s)u^4=0
$$
Let $W$ and $U$ be elements of $k[x,y]$ such that $W+\h=w$ and
$U+\h=u$. Then, thanks to the previous equation:
\begin{equation}\label{Ecu 1}
W^2+(t+s)U^4=hG
\end{equation}
for some $G\in k[x,y]$. By applying the partial derivatives
$\partial_s$ and $\partial_t$ to (\ref{Ecu 1}), we obtain:
$$
\begin{array}{rc}
\partial_t:&U^4=x^4G+h\partial_t(G)\\
\partial_s:&U^4=y^4G+h\partial_s(G).
\end{array}
$$
Then, if $g:=G+\h$, we have the following equalities in $A$:
$$
\left.
\begin{array}{rc}
\partial_t:&u^4=x^4g\\
\partial_s:&u^4=y^4g
\end{array} \right\} \Rightarrow (x^4-y^4)g=0.
$$
Since $A$ is a domain and $x^4\neq y^4$, $g=0$, so $u=u_1=v_1=0$.
Then, we can not integrate any non-zero $k$-derivation until length 4, i.e.
$\IDer_k(A;4)=0$ and $L\otimes_k\IDer_k(A;4)=0$.
\medskip

To prove that $\IDer_L(A_L;4)$ is not zero, we calculate
$\IDer_L(\log \h^e;4)$. Thanks to Proposition \ref{I^p
-logarithmicas}, it is enough to calculate $\IDer_L(\log H;2)$ where
$H=x+y+t^{1/2}x^2+s^{1/2}y^2$. Note that $J^0=\langle 1\rangle$ so,
by Proposition \ref{Narvaez Prop 2.2.4}, any $I$-logarithmic
$k$-derivation is integrable. It is easy to see that $\Der_L(\log
H)=\langle
\partial_x+\partial_y,H\partial_x\rangle$. Hence, thanks to
Proposition \ref{Anillo polinomios, logaritmico, map sobrey},
$\IDer_L(A_L;4)=\langle
\overline\delta_1,\overline\delta_2\rangle\neq 0$ where
$\overline\delta_1$ (resp. $\overline \delta_2$) is the derivation
induced by $\partial_x+\partial_y$ (resp.
$H\partial_x$) in the quotient. Therefore, $\Phi_4^{L,A}$
is not surjective.
\vspace{-0.5cm}\begin{flushright}$\square$\end{flushright}

We have seen that $\Phi_m^{L,A}$ is not surjective in general,
however, if we assume that $L$ is not only flat, but satisfies some
additional conditions, then $\Phi^{L,A}_m$ will be surjective for all $m\geq 1$
and any finitely generated $k$-algebra $A$.

\subsubsection{The extension is a polynomial ring}

In this section, we assume that $k$ is a commutative ring and
$L:=k[t_i\ |\ i\in \I]$ is a polynomial ring in an arbitrary number
of variables. We define $\N^{(\I)}=\{\alpha:=(\alpha_i)_{i\in \I}\
|\ \alpha_i=0 \mbox{ \it  except for a finite number of } i\in \I\}$
and, for $\alpha\in \N^{(\I)}$, we put $t^\alpha=\prod_{i\in \I}
t_i^{\alpha_i}$. We start with some numerical results.

The following lemma is clear.
\begin{lem}\label{Parte entera} Let $n\leq m$ be two positive integers. We have the following properties.
\begin{itemize}
\item[a.] $\left(\lfloor m/n \rfloor+1\right)n-1\geq m$.
\item[b.] If $m\neq 0\mod n$, then $\lfloor m/n\rfloor=\lfloor
(m-1)/n\rfloor$. Otherwise, $\lfloor m/n\rfloor=\lfloor
(m-1)/n\rfloor+1$.
\item[c.] If $n<m$ such that $m=0\mod n$. Then, there exists a prime
factor of $m$ which divides $m/n$.
\end{itemize}
\end{lem}

\begin{Def}\label{Pn}
Let $n$ be a positive integer. We define
$$
\PR_n=\bigcup_{q \mbox{ factor prime of } n }q\N^{(\I)}.
$$
\end{Def}

\begin{lem}\label{Conjunto disjunto}
Let $n,s$ be two positive integers such that $n\neq s$. Then, there
do not exist $\alpha\in \N^{(\I)}\setminus \PR_n$ and $\eta\in
\N^{(\I)}\setminus \PR_s$ such that $\alpha s=\eta n$.
\end{lem}

\begin{proof} Suppose that there exist $\alpha\in
\N^{(\I)}\setminus \PR_n$ and $\eta\in \N^{(\I)}\setminus \PR_s$ such
that $\alpha s=\eta n$. If there were such a prime that divides $n$
and $s$, then we could simplify it. So, we can assume that $s$ and
$n$ do not have prime factors in common. Now, as $s$ and $n$ are not
the same, one of them, we say $s$, has a prime factor $q$ such that
does not divide to the another one, in this case $n$. Since $\alpha
s=\eta n$, we have that $\alpha_i s=\eta_in$ for all $i\in \I$. So,
$q$ divide $\eta_i$ for all $i\in \I$. Then $\eta=q\eta'\in \PR_s$
and we have a contradiction.
\end{proof}

Fix $m>1$ an integer and consider $m=q_1^{a_1}\cdots q_s^{a_s}$ its prime
factorization, i.e. for all $j=1,\ldots s$, $q_j$ is a prime,
$a_j>0$ and  $q_j\neq q_i$ if $i\neq j$. Let us consider $\beta\in
\PR_m$. Then, we can write $\beta=q_1^{b_1}\cdots q_s^{b_s}\eta$
where $b_j\geq 1$ for some $j\in\{1,\ldots, s\}$ and $\eta\in \N^{(\I)}$
such that $q_j\nmid \eta$ for any $j=1,\ldots,s$, i.e. for all $j$
there exists $\eta_{i_j}$ with $i_j\in \I$ such that $q_j\nmid
\eta_{i_j}$. We can assume, without loss of generality, that there
exists an integer $l_\beta$ such that $0\leq l_\beta\leq s$ and
$a_j>b_j$ for all $j\leq l_\beta$ and $a_j\leq b_j$ for all
$j>l_\beta$. Then, we define
$$
n_\beta=\left\{
\begin{array}{lcc}
1&\mbox{if } l_\beta=0\\
q_1^{a_1-b_1}\cdots q_{l_\beta}^{a_{l_\beta}-b_{l_\beta}}&\mbox{if
}l_\beta\geq 1.
\end{array}
\right.
$$

\begin{lem}\label{Existencia de una n}
For each $\beta\in \PR_m$, there exists a unique $n\in \N$ with
$1\leq n<m$ such that $m=0\mod n$ and $\beta n/m\not\in\PR_n$.
\end{lem}

\begin{proof} We write $\beta=q_1^{b_1}\cdots
q_s^{b_s}\eta$, where $\eta\in \N^{(\I)}$ such that $q_j\nmid\eta$
for any $j=1,\ldots,s$ and $b_j\geq 1$ for some $j\in \{1,\ldots,s\}$.
We take $n=n_\beta$. It is obvious that $n$ divides $m$ and $1\leq
n<m$. We denote $l:=l_\beta$ to simplify the notation. We put
$$
\alpha:=\dfrac{\beta n}{m}=\dfrac{\eta q_1^{b_1}\cdots
q_s^{b_s}n}{q_1^{a_1}\cdots q_s^{a_s }}.
$$
If $l=0$, then $n=1$ and $\PR_1=\emptyset$ so, $\alpha\not\in\PR_n$
(note that $\alpha\in \N^{(\I)}$ because if $l=0$, then $b_j\geq
a_j$ for all $j=1,\ldots,s$). If $l\geq 1$, then
$$
\alpha=\dfrac{\eta q_1^{b_1}\cdots q_s^{b_s} q_1^{a_1-b_1}\cdots
q_l^{a_l-b_l}}{q_1^{a_1}\cdots q_s^{a_s}}= \dfrac{\eta
q_1^{a_1}\cdots q_l^{a_l} q_{l+1}^{b_{l+1}}\cdots
q_s^{b_s}}{q_1^{a_1}\cdots
q_s^{a_s}}=q_{l+1}^{b_{l+1}-a_{l+1}}\cdots q_s^{b_s-a_s}\eta.
$$
Note that set of primes which divide $n$ is
$\{q_1,\ldots,q_l\}$. Hence, $q_j\nmid \alpha$ for all $j=1,\ldots,l$
(recall that $q_j\nmid\eta$). So, $\alpha\not\in \PR_n$.

Now, let us suppose that there exists another $n'\in \N$ holding the
lemma, in particular $\alpha':=\beta n'/m\not\in \PR_{n'}$. Then,
$\alpha n'= \alpha' n$ and we have a contradiction by Lemma
\ref{Conjunto disjunto}.
\end{proof}


\begin{teo}\label{Expresion de D en RL}
Let $m\geq 1$ be an integer and $L=k[t_i\ |\mbox{ }i\in \I]$ a
polynomial ring. Let us consider $D\in \HS_L(R_L;m)$. Then, for all
$n=1,\ldots, m$ there exist a finite subset $L_n$ of
$\N^{(\I)}\setminus \PR_n$ and $N^{n,\alpha}\in \HS_k(R)$ for each
$\alpha\in L_n$ such that
$$
D=\circ_{n=1}^m\left(\circ_{\alpha\in
L_n}\left(\psi_{\alpha}^{n,m}\bullet \widetilde
{N^{n,\alpha}}\right)\right)
$$
where $\psi^{n,m}_\alpha:R_L[|\mu|]\rightarrow R_L[|\mu|]_m$ is the
substitution map of constant coefficients given by
$\psi_\alpha^{n,m}(\mu)=t^\alpha\mu^n$.
\end{teo}

\begin{proof} First, observe that, if $E\in \HS_L(R_L;m)$
then, $\psi_\alpha^{n,m} \bullet
E=\tau_{\infty,m}\left(\left(t^\alpha\bullet E\right)[n]\right)$. We
prove this theorem by induction on $m$. Assume that $m=1$ then,
$D=(\Id,D_1)\in \HS_L(R_L;1)$. Since $L$ is free over $k$ and
$\{t^\alpha\mbox{ }|\mbox{ } \alpha\in \N^{(\I)}\}$ is a $k$-basis
of $L$, from the proof of Proposition \ref{prop:free-extension-der},
$D_1\in \Der_L(R_L)$ can be written as $D_1=\sum_{\alpha\in L_1}
t^\alpha \widetilde{\delta_\alpha}$ where $L_1$ is a finite subset of $\N^{(\I)}$ and $\delta_\alpha\in
\Der_k(R)$ for each $\alpha\in L_1$. Let us
consider $N^{1,\alpha}$ an integral of $\delta_\alpha$ for
$\alpha\in L_1$. Then, $\widetilde{N^{1,\alpha}}\in \HS_L(R_L)$ is
an integral of $\widetilde\delta_\alpha$ (see Corollary
\ref{Integrales_y_extension}). Hence,
$$
D=\circ_{\alpha\in L_1} \left(t^\alpha\bullet (\Id,
\widetilde{\delta_\alpha})\right)=\circ_{\alpha\in L_1}
\left(\tau_{\infty,1}\left(t^\alpha\bullet
\widetilde{N^{1,\alpha}}\right)\right)=\circ_{\alpha\in
L_1}\left(\psi^{1,1}_\alpha \bullet \widetilde{N^{1,\alpha}}\right)
$$
(note that the order of the composition in this equality is not
important because $\HS_L(R_L;1)\equiv \Der_L(R)$ is a commutative
group) and we have the result when $m=1$. Let us assume that the
theorem is true for any HS-derivation of length $m-1$ and we will
prove it for $D\in \HS_L(R_L;m)$. By induction hypothesis, for all
$n=1,\ldots,m-1$, there exist a finite subset $L_n'$  of
$\N^{(\I)}\setminus \PR_n$ and  $N^{n,\alpha}\in \HS_k(R)$ for all
$\alpha\in L'_n$ such that
\begin{equation}\label{Pol. Extension 1}
\tau_{m,m-1}(D)=\circ_{n=1}^{m-1} \left(\circ_{\alpha\in L_n'}
\left(\psi^{n,m-1}_\alpha \bullet
\widetilde{N^{n,\alpha}}\right)\right).
\end{equation}
We define
$$
E:=\circ_{n=1}^{m-1}\left(\circ_{\alpha\in L'_n}
\left(\psi_\alpha^{n,m}\bullet
\widetilde{N^{n,\alpha}}\right)\right)
$$
where the composition is taken in the same order that in (\ref{Pol.
Extension 1}). Note that $\psi^{n,m-1}_\alpha=\tau_{m,m-1}\circ
\psi^{n,m}_\alpha$, and thanks to Lemma \ref{composicion y
sustituciones} and Corollary \ref{La truncation y el desplazamiento
iterado}, we have that:
$$
\begin{array}{rl}
\tau_{m,m-1}(E)=&\displaystyle
\circ_{n=1}^{m-1}\left(\circ_{\alpha\in L'_n}
\left(\tau_{m-1}\bullet\left(\psi_\alpha^{n,m}\bullet
\widetilde{N^{n,\alpha}}\right)\right)\right)
=\displaystyle\circ_{n=1}^{m-1}\left(\circ_{\alpha\in L'_n}
\left(\left(\tau_{m,m-1}\circ\psi_\alpha^{n,m}\right)\bullet
\widetilde{N^{n,\alpha}}\right)\right)=\tau_{m,m-1}(D).
\end{array}
$$
Then, by Lemma \ref{La ultima componente difiere en una derivacion},
$D=E\circ (\Id,\delta)[m]$ where $\delta\in \Der_L(R_L)$. From the
proof of Proposition \ref{prop:free-extension-der},
$\delta=\sum_{\beta\in \J} t^\beta \widetilde{\delta_\beta}$ where
$\J$ is a finite subset of $\N^{(\I)}$ and $\delta_\beta\in
\Der_k(R)$ for all $\beta\in \J$. We denote $\Gamma=\{n\in \N\mbox{
}|\mbox{ } 1\leq n\leq m-1, \mbox{ } m=0\mod n\}$. For all $n \in
\Gamma$, we define
$$
\begin{array}{ccc}
\J_n:=\left\{\beta\in \J\mbox{ }|\mbox{ } \beta=\alpha (m/n) \mbox{
for some }\alpha \in L'_n\right\}&\mbox{ and }& \mathcal
L_m=\J\setminus \PR_m.
\end{array}
$$

{\bf Claim 1.} {\it For all $n,s\in \Gamma$ such that $n\neq s$,
then $\J_n\cap \J_s=\emptyset$.}

\smallskip

 Let us suppose that there exists
$\beta\in \J_n\cap\J_s$. In this case, there exist $\alpha\in
L'_n\subseteq \N^{(\I)} \setminus \PR_n$ and $\eta\in L'_s\subseteq
\N^{(\I)}\setminus \PR_s$ such that $\beta=\alpha(m/n)=\eta(m/s)$,
i.e. $\alpha s=\eta n$ and this can not happen by Lemma
\ref{Conjunto disjunto}. \medskip

{\bf Claim 2.} {\it $\mathcal L_m\cap \J_n=\emptyset$ for all $n\in
\Gamma$.}

\smallskip

By Lemma \ref{Parte entera} c., there exists a prime factor, $q$, of
$m$ that divides $m/n$. Assume that $\beta\in \J_n$. Then, we have
that $\beta=\alpha(m/n)$ for some $\alpha\in L_n'$. Then, $q|\beta$
so, $\beta\in \PR_m$.
\medskip

Let us write $\J=\sqcup_{n\in \Gamma} \J_n \sqcup \mathcal L_m
\sqcup \overline\J$ where $\overline\J=\J\setminus\left(\sqcup_{n\in
\Gamma}\J_n\sqcup \mathcal L_m \right)$. Observe that
$\overline\J\subseteq \PR_m$ so, from Lemma \ref{Existencia de una
n}, for all $\beta\in \overline\J$, there exists a unique
$n_\beta\in \Gamma$ such that $(\beta n_{\beta})/m\not\in
\PR_{n_\beta}$. Therefore, if we denote $\overline\J_n=\{\beta\in
\overline\J\mbox{ }|\mbox{ } n_\beta=n\}$ for all $n\in \Gamma$, we
can write
$$
\J=\sqcup_{n\in \Gamma} \left(\J_n \sqcup \overline\J_n\right)
\sqcup \mathcal L_m \mbox{ \hspace{0.4cm} and
\hspace{0.4cm}}\delta=\sum_{n\in \Gamma} \sum_{\beta\in \J_n \sqcup
\overline\J_n} t^\beta \delta_\beta+\sum_{\alpha\in \mathcal L_m}
t^\alpha \delta_\alpha.
$$
Now, for each $n\in \Gamma$ we can define
$$\begin{array}{lcc}
\mathcal L'_n=\{\alpha \in L_n'\ |\mbox{ } \alpha(m/n)\in
\J_n\}&\mbox{ and }& \overline{\mathcal L}_n=\{\alpha\in
\N^{(\I)}\setminus L_n'\ |\mbox{ } \alpha(m/n)\in \overline
\J_n\}\nsubseteq \PR_n.
\end{array}
$$
Note that $\mathcal L'_n\cap \overline{\mathcal L}_n=\emptyset$. Let
us denote  $\mathcal L_n=\mathcal L'_n \cup \overline{\mathcal
L}_n$. Hence, we can express
$$
(\Id,\delta)=\circ_{n\in \Gamma} \left(\circ_{\alpha\in \mathcal
L'_n} \left(\Id,
t^{\alpha(m/n)}\widetilde{\delta_{\alpha(m/n)}}\right)
\circ_{\alpha\in\overline{\mathcal
L}_n}\left(\Id,t^{\alpha(m/n)}\widetilde{\delta_{\alpha(m/n)}}\right)\right)
\circ \left(\circ_{\alpha\in \mathcal L_m}
\left(\Id,t^\alpha\widetilde{\delta_\alpha}\right)\right).
$$
By Corollary \ref{La truncation y el desplazamiento iterado} and
Lemma \ref{Relaciones entre operaciones}, for each $n\in \Gamma \cup
\{m\}$ and $\alpha\in \mathcal L_n$, we have that:
$$
\left(\Id,t^{\alpha(m/n)}
\widetilde{\delta_{\alpha(m/n)}}\right)[m]= \left(\left(
t^{\alpha(m/n)}\bullet
\left(\Id,\widetilde{\delta_{\alpha(m/n)}}\right)\right)[m/n]\right)[n]=
\left(t^\alpha \bullet
\left(\left(\Id,\widetilde{\delta_{\alpha(m/n)}}\right)[m/n]\right)
\right)[n].
$$
For each $n\in \Gamma \cup \{m\}$ and $\alpha\in \mathcal L_n$, let
us consider $M^{n,\alpha}\in \HS_k(R)$ an integral of
$\delta_{\alpha(m/n)}$.  We know that $\widetilde{M^{n,\alpha}}$ is
an integral of $\widetilde{\delta_{\alpha(m/n)}}$, so
$\widetilde{M^{n,\alpha}}[m/n]$ is an integral of
$\left(\Id,\widetilde{\delta_{\alpha(m/n)}}\right)[m/n]$. Hence, by
Lemma \ref{Relaciones entre operaciones}, we have
$$
\begin{array}{rl}
\psi^{n,m}_{\alpha} \bullet
\left(\widetilde{M^{n,\alpha}}[m/n]\right)=&\displaystyle
\tau_{\infty,m}\left(\left(t^\alpha\bullet
\left(\widetilde{M^{n,\alpha}}[m/n]\right)\right)[n]\right)=
\left(\tau_{\infty,m/n}\left(t^\alpha\bullet
\left(\widetilde{M^{n,\alpha}}[m/n]\right)\right)\right)[n]\\[0.2cm]
=&\displaystyle\left(t^\alpha \bullet
\tau_{\infty,m/n}\left(\widetilde{M^{n,\alpha}}[m/n]\right)\right)[n]=\left(t^\alpha
\bullet
\left(\left(\Id,\widetilde{\delta_{\alpha(m/n)}}\right)[m/n]\right)
\right)[n].
\end{array}
$$
To simplify the following expression, we put $\overline{\mathcal
L}_n=\mathcal L'_n=\emptyset$ for all $n\in
\{1,\ldots,m-1\}\setminus \Gamma$. Moreover, for all $n\in
\{1,\ldots,m-1\}$, if $\alpha\in L_n'\setminus \mathcal L'_n$ then
we consider $\delta_{\alpha(m/n)}=0$ and $M^{n,\alpha}=\II\in
\HS_k(R)$ an integral of $\delta_{\alpha(m/n)}$. Thanks to Lemmas
\ref{Composicion conmuta con las derivaciones en la ultima
componente} and \ref{composicion y sustituciones} and the previous
equation, we can write:
$$
\begin{array}{rl}
D=&\displaystyle\circ_{n=1}^{m-1}\left(
 \circ_{\alpha\in L'_n}
  \left(\psi^{n,m}_\alpha\bullet \widetilde{N^{n,\alpha}} \circ
  \left(\Id,t^{\alpha(m/n)}\widetilde{\delta_{\alpha(m/n)}}\right)[m]\right)
  \circ \left(\circ_{\alpha\in \overline{\mathcal L}_n}
  \left(\Id,t^{\alpha(m/n)}\widetilde{\delta_{\alpha(m/n)}}\right)[m]\right)
  \right)
  \circ\\[0.2cm]&\displaystyle\circ
  \left(\circ_{\alpha\in \mathcal L_m}
  \left(\Id,t^\alpha\widetilde{\delta_\alpha}\right)[m]\right)=\\[0.2cm]
=&\displaystyle\circ_{n=1}^{m-1}\left(
 \circ_{\alpha\in L'_n}
   \left(\psi^{n,m}_\alpha\bullet \widetilde{N^{n,\alpha}} \circ
    \psi^{n,m}_\alpha \bullet (\widetilde{M^{n,\alpha}}[m/n])\right)
  \circ_{\alpha\in \overline{\mathcal L}_n} \left(
  \psi^{n,m}_\alpha \bullet (\widetilde{M^{n,\alpha}}[m/n])\right)\right)\\&
\displaystyle\circ \left(\circ_{\alpha\in \mathcal L_m}
\left(\psi^{m,m}_\alpha \bullet
\widetilde{M^{m,\alpha}}\right)\right)\\[0.2cm]
=&\displaystyle\circ_{n=1}^{m-1}\left(
 \circ_{\alpha\in L'_n}
   \left(\psi^{n,m}_\alpha\bullet \left(\widetilde{N^{n,\alpha}} \circ
    \widetilde{M^{n,\alpha}}[m/n]\right)\right)
  \circ_{\alpha\in \overline{\mathcal L}_n} \left(
  \psi^{n,m}_\alpha \bullet (\widetilde{M^{n,\alpha}}[m/n])\right)\right)
\circ \left(\circ_{\alpha\in \mathcal L_m} \left(\psi^{m,m}_\alpha
\bullet
\widetilde{M^{m,\alpha}}\right)\right).\\
\end{array}
$$
Thanks to Lemma \ref{Propiedades Extension} (2),
$\widetilde{M^{n,\alpha}}[m/n]$ is the extension of the
HS-derivation $M^{n,\alpha}[m/n]$ and, by Lemma \ref{Propiedades
Extension} (1), $\widetilde{N^{n,\alpha}}\circ
\widetilde{M^{n,\alpha}}[m/n]$ is the extension of $N^{n,\alpha}
\circ M^{n,\alpha}[m/n]$. Therefore, if we denote $L_n=L_n'\cup
\overline{\mathcal L}_n\subseteq \N^{(\I)}\setminus \PR_n$ and
$L_m=\mathcal L_m$, we have the theorem.
\end{proof}

\begin{teo}\label{Expresion de D, trascen, I-log}
Let $m\geq 1$ be an integer, $L=k[t_i\mbox{ }|\mbox{ }i\in \I]$ a
polynomial ring, $I\subseteq R$ an ideal and $D\in \HS_L(\log
I^e;m)$. For all $n=1,\ldots, m$, let $L_n$ be a finite subset of
$\N^{(\I)}\setminus \PR_n$ and $N^{n,\alpha}\in \HS_k(R)$ for all
$\alpha\in L_n$ such that
$$ D=\circ_{n=1}^m\left(\circ_{\alpha\in
L_n}\left(\psi_{\alpha}^{n,m}\bullet \widetilde
{N^{n,\alpha}}\right)\right)
$$
where $\psi_\alpha^{n,m}:R_L[|\mu|]\to R_L[|\mu|]_m$ is the
substitution map given by $\psi^{n,m}_\alpha(\mu)=t^\alpha \mu^n$.
Then, for all $n=1,\ldots, m$ and $\alpha\in L_n$, $N^{n,\alpha}\in
\HS_k(R)$ is an $\lfloor m/n\rfloor-I$-logarithmic HS-derivation.
\end{teo}

\begin{proof} We prove this result by induction on $m$. If
$m=1$, we have to prove that $N^{1,\alpha}$ is $1-I$-logarithmic for
all $\alpha\in L_1$, i.e. $N^{1,\alpha}_1\in \Der_k(\log I)$ for all
$\alpha\in L_1$. In this case,
$$
D=\circ_{\alpha\in L_1} \left(\psi_\alpha^{1,1} \bullet \widetilde
{N^{1,\alpha}}\right)=\circ_{\alpha\in L_1}
\left(\tau_{\infty,1}\left(t^\alpha\bullet
\widetilde{N^{1,\alpha}}\right)\right)=\circ_{\alpha\in L_1}
\left(\Id,t^\alpha\left(\widetilde{N^{1,\alpha}}\right)_1\right)
\Rightarrow D_1=\sum_{\alpha\in L_1} t^\alpha
\left(\widetilde{N^{1,\alpha}}\right)_1.
$$
Since $D_1$ is $I^e$-logarithmic, doing the same process of the
proof of Proposition \ref{prop:free-extension-der}, we have that
$N^{n,\alpha}$ is $1-I$-logarithmic. Assume that the theorem is true
for all $I^e$-logarithmic HS-derivation of length $m-1$ and let us
take $D\in \HS_L(\log I^e;m)$ such that
$$
D=\circ_{n=1}^m\left(\circ_{\alpha\in L_n}\left(\psi^{n,m}_\alpha
\bullet \widetilde{N^{n,\alpha}}\right)\right)
$$
where $L_n\subseteq \N^{(\I)}\setminus \PR_n$ is a finite set and
$N^{n,\alpha}\in \HS_k(R)$ for all $\alpha\in L_n$ and
$n=1,\ldots,m$. By Corollary \ref{La truncation y el desplazamiento
iterado}, we have that
$$
\tau_{m,m-1}(D)=\circ_{n=1}^{m-1} \left(\circ_{\alpha\in
L_n}\tau_{m,m-1}\bullet\left( \psi^{n,m}_\alpha \bullet
\widetilde{N^{n,\alpha}}\right)\right)\circ\left(\circ_{\alpha\in
L_m} \tau_{m,m-1}\bullet\left(\psi^{m,m}_\alpha\bullet
\widetilde{N^{n,\alpha}}\right)\right).
$$
From Lemma \ref{composicion y sustituciones}, for any $E\in
\HS_L(R_L)$,  $\tau_{m,m-1} \bullet (\psi^{n,m}_\alpha
\bullet(E))=(\tau_{m,m-1}\circ \psi^{n,m}_\alpha)\bullet
E=\psi^{n,m-1}_\alpha\bullet E$. Moreover,
$\psi^{m,m-1}_\alpha\bullet E=\II$. So,
$$
\tau_{m,m-1}(D)=\circ_{n=1}^{m-1} \left(\circ_{\alpha\in L_n}
\psi^{n,m-1}_\alpha\bullet \widetilde{N^{n,\alpha}}\right).
$$
Hence, since $\tau_{m,m-1}(D)\in \HS_L(\log I^e;m-1)$,  we can apply
the induction hypothesis and we deduce that  $N^{n,\alpha}\in
\HS_k(R)$ is $\lfloor (m-1)/n\rfloor-I$-logarithmic for all
$\alpha\in L_n$ and $n=1,\ldots,m-1$. We define
$$
E^n:=\circ_{\alpha\in L_n} \left(\psi^{n,m}_\alpha \bullet
\widetilde{N^{n,\alpha}}\right) \Rightarrow D=E^1\circ \cdots \circ
E^m
$$
where the order of the composition in $E^n$ is the same that in $D$.
\smallskip

 {\bf Claim.} {\it $E^n$ is $(m-1)-I^e$-logarithmic:}
Since $N^{n,\alpha}$ is $\lfloor (m-1)/n\rfloor-I$-logarithmic, by
Lemma \ref{Propiedades Extension} (3),  $t^\alpha\bullet
\widetilde{N^{n,\alpha}}$ is
$\lfloor(m-1)/n\rfloor-I^e$-logarithmic. From Lemma
\ref{lema-resumen} (a), $\left(t^\alpha\bullet
\widetilde{N^{n,\alpha}}\right)[n]$ is
$\left((\lfloor(m-1)/n\rfloor+1)n-1\right)-I^e$-logarithmic. By
Lemma \ref{Parte entera} a., $m-1< (\lfloor(m-1)/n\rfloor +1)n-1$,
so $\psi^{n,m}_\alpha \bullet \widetilde{N^{n,\alpha}}$ is
$(m-1)-I^e$-logarithmic because  $\psi^{n,m}_\alpha\bullet
\ast=\tau_{\infty,m}\left((t^\alpha\bullet \ast)[n]\right)$. Hence,
by Lemma \ref{lema-resumen} (d), $E^n$ is $(m-1)-I^e$-logarithmic
for all $n$.
\medskip

Let us consider $n\in \{1,\ldots,m\}$ such that $n \nmid m$. Then,
by Corollary \ref{La truncation y el desplazamiento iterado},
$$
E^n=\circ_{\alpha\in L_n} \left(\psi^{n,m}_\alpha\bullet
\widetilde{N^{n,\alpha}}\right)=\circ_{\alpha\in L_n}
\tau_{\infty,m}\left(\left(t^\alpha \bullet
\widetilde{N^{n,\alpha}}\right)[n]\right)=
\tau_{\infty,m}\left(\left(\circ_{\alpha\in L_n}
\left(t^\alpha\bullet
\widetilde{N^{n,\alpha}}\right)\right)[n]\right).
$$
Hence, $E^n_m=0$. Moreover, by Lemma
\ref{Parte entera} b., $\lfloor (m-1)/n\rfloor=\lfloor m/n\rfloor$,
so $N^{n,\alpha}$ is $\lfloor m/n \rfloor-I$-logarithmic. Therefore,
to prove the theorem we have to show that $N^{n,\alpha}$ is
$(m/n)-I$-logarithmic for $n|m$. By Lemma \ref{Parte entera} b.,
$m/n=\lfloor (m-1)/n\rfloor +1$ and, since $N^{n,\alpha}$ is
$\lfloor (m-1)/n\rfloor-I$-logarithmic, it is enough to prove that
$N^{n,\alpha}_{m/n}(I)\subseteq I$. Note that
$$
\left(\psi^{n,m}_\alpha \bullet
\widetilde{N^{n,\alpha}}\right)_m=\left(\tau_{\infty,m}
\left(\left(t^\alpha\bullet
\widetilde{N^{n,\alpha}}\right)[n]\right)\right)_m=t^{\alpha(m/n)}
\left(\widetilde{N^{n,\alpha}}\right)_{m/n}
$$
where
${\left(\widetilde{N^{n,\alpha}}\right)_{m/n}}_{|R}=N^{n,\alpha}_{m/n}$.
Therefore, by Lemma \ref{lema-resumen} (d),
$$
E^n_m=\sum_{\alpha\in L_n}
t^{\alpha(m/n)}(\widetilde{N^{n,\alpha}})_{m/n}+F_n
$$
where $F_n$ is an $I^e$-differential operator. Hence, again by Lemma
\ref{lema-resumen} (d),
$$
D_m=\sum_{n=1}^m E^n_m +F=\sum_{n|m} \sum_{\alpha\in L_n}
t^{\alpha(m/n)}(\widetilde{N^{n,\alpha}})_{m/n}+F_n+F
$$
where $F$ is an $I^e$-differential operator. Since $D_m$ is also an
$I^e$-differential operator, we have that
$$
\sum_{n|m} \sum_{\alpha\in L_n}
t^{\alpha(m/n)}(\widetilde{N^{n,\alpha}})_{m/n} \mbox{ is an
$I^e$-differential operator.}
$$
Observe that $\alpha(m/n)\neq \eta(m/s)$ for all $\alpha\in L_n$ and
$\eta\in L_s$ by Lemma \ref{Conjunto disjunto} because $L_n\subseteq
\N^{(\I)}\setminus \PR_n$ and $L_s\in \N^{(\I)}\setminus \PR_s$.
Doing the same process than in the proof of Proposition
\ref{prop:free-extension-der}, we can deduce that $N^{n,\alpha}\in
\HS_k(R)$ is $\lfloor m/n\rfloor-I$-logarithmic for all $\alpha\in
L_n$ and $n=1,\ldots,m$.
\end{proof}

\begin{teo}
Let $m\geq 1$ be an integer, $L=k[t_i\mbox{ } |\mbox{ } i\in \I]$ a
polynomial ring, $A$ a finitely generated $k$-algebra and $E\in
\HS_L(A_L;m)$. Then, for all $n=1,\ldots,m$ there exist  a finite
subset $L_n\subseteq \N^{(\I)}\setminus \PR_n$ and $M^{n,\alpha}\in
\HS_k(A;\lfloor m/n\rfloor)$ for each $\alpha \in L_n$ such that
$$
E=\circ_{n=1}^m \left(\circ_{\alpha\in L_n}\left(\phi^{n,m}_\alpha
\bullet \widetilde{M^{n,\alpha}}\right)\right)
$$
where $\phi^{n,m}_\alpha:A_L[|\mu|]_{\lfloor m/n \rfloor}\to
A_L[|\mu|]_m$ is the substitution map of constant coefficients given
by $\phi^{n,m}_\alpha(\mu)=t^\alpha \mu^n$.
\end{teo}

\begin{proof}
Since $A$ is a finitely generated $k$-algebra, we can take $A=R/I$
where $R=k[x_1,\ldots,x_d]$ and $I\subseteq R$ an ideal. By
Proposition \ref{HS log es sobreyectivo}, there exists $D\in
\HS_k(\log I^e;m)$ such that $\Pi_{\HS;m}^{I^e}(D)=E$. By theorems
\ref{Expresion de D en RL} and  \ref{Expresion de D, trascen,
I-log}, for all $n=1,\ldots,m$ there exist a finite subset $L_n$ of
$\N^{(\I)}\setminus \PR_n$ and an $\lfloor m/n\rfloor-I$-logarithmic
HS-derivation $N^{n,\alpha}\in \HS_k(R)$ such that
$$
D=\circ_{n=1}^m \left(\circ_{\alpha\in L_n} \left(\psi^{n,m}_\alpha
\bullet \widetilde{N^{n,\alpha}}\right)\right)
$$
where $\psi^{n,m}_\alpha:R_L[\mu|]\to R_L[|\mu|]_m$ is the
substitution map given by $\psi^{n,m}_\alpha(\mu)=t^\alpha \mu^n$.

\smallskip

Let us consider $\theta^{n,m}_\alpha:R_L[|\mu|]_{\lfloor
m/n\rfloor}\to R_L[|\mu|]_m$ the substitution map given by
$\theta^{n,m}_\alpha(\mu)=t^\alpha \mu^n$. Then,
$\psi^{n,m}_\alpha=\theta^{n,m}_\alpha\circ \tau_{\infty,\lfloor
m/n\rfloor}$. So, let us rewrite $N^{n,\alpha}=\tau_{\infty,\lfloor
m/n \rfloor}(N^{n,\alpha})\in \HS_k(\log I; \lfloor m/n\rfloor)$ and
we have that
$$
D=\circ_{n=1}^m \left(\circ_{\alpha\in L_n} \left( \theta^{n,m}_n
\bullet \widetilde{N^{n,\alpha}}\right)\right)
$$
(note that $\widetilde{\tau_{\infty,s}(N})=\tau_{\infty,s}(\widetilde{N})$ for any $N\in
\HS_k(R;m)$ and $s\geq 1$ by Lemma \ref{Propiedades Extension}).
Moreover $\phi^{n,m}_\alpha$ is the induced map by
$\theta^{n,m}_\alpha$ in $A_L$.  Therefore, by Lemmas \ref{Susti y
Poli} and \ref{Extension y poli},
$$
\begin{array}{rl}
E=\Pi_{\HS,m}^{I^e}(D)=&\displaystyle  \circ_{n=1}^m
\left(\circ_{\alpha\in L_n} \left(\Pi_{\HS,m}^{I^e} \left(
\theta^{n,m}_\alpha \bullet
\widetilde{N^{n,\alpha}}\right)\right)\right)= \circ_{n=1}^m
\left(\circ_{\alpha\in L_n}\left( \phi^{n,m}_\alpha \bullet
\left(\Pi^{I^e}_{\HS, \lfloor
m/n\rfloor}(\widetilde{N^{n,\alpha}})\right)\right)\right)\\[0.2cm]
 =&\displaystyle \circ_{n=1}^m
\left(\circ_{\alpha\in L_n}\left( \phi^{n,m}_\alpha \bullet
\left(\widetilde{M^{n,\alpha}}\right)\right)\right)
\end{array}
$$
where $\widetilde{M^{n,\alpha}}\in \HS_L(A_L;m)$ is the extension of
$\Pi^{I}_{\HS, \lfloor m/n\rfloor}(N^{n,\alpha})\in \HS_k(A; \lfloor
m/n\rfloor)$ and the theorem is proved.

\end{proof}

\begin{cor}\label{Anillo polinomios, algebra, sobrey}
Let $k$ be a ring, $L=k[t_i\ |\mbox{ }i\in \I]$ and $A$ a finitely
generated $k$-algebra. We denote $A_L=A\otimes_k L$. Then,
$\Phi_m^{L,A}:L\otimes \IDer_k(A;m)\rightarrow \IDer_L(A_L;m)$ is an
isomorphism of $A_L$-modules for all $m\in \N$. Moreover,
$\Leap_k(A)=\Leap_L(A_L)$.
\end{cor}

\begin{proof}
Since $L$ is flat over $k$, from Lemma
\ref{Inyectividad/Sobreyectividad}, $\Phi_m^{L,A}$ is injective. To
prove the surjectivity, we take $\delta\in \IDer_L(A_L;m)$. By
definition of integrability, there exists $E\in \HS_L(A_L;m)$ such
that $E_1=\delta$. By the previous theorem, we can write $E$ as
$$
E=\circ_{n=1}^m \left(\circ_{\alpha\in L_n}\left(\phi^{n,m}_\alpha
\bullet \widetilde{M^{n,\alpha}}\right)\right)
$$
where, for all $n=1,\ldots,m$, $L_n$ is a finite subset of
$\N^{(\I)}\setminus \PR_n$ and, for all $\alpha\in L_n$,
$M^{n,\alpha} \in \HS_k(A;\lfloor m/n \rfloor)$ and
$\phi^{n,m}_\alpha:A_L[|\mu|]_{\lfloor m/n\rfloor}\to A_L[|\mu|]_m$
is the substitution map given by
$\phi^{n,m}_{\alpha}(\mu)=t^\alpha\mu^n$. If $n>1$, then
$\ell\left(\phi^{n,m}_\alpha \bullet N\right)>1$ for all $N\in
\HS_L(A_L;m)$ and if $n=1$, then $M^{1,\alpha}_1\in \IDer_k(A;m)$.
Hence,
$$
\delta=E_1=\left(\circ_{\alpha\in L_1} \left(\phi^{1,m}_\alpha
\bullet \widetilde{M^{n,\alpha}}\right)\right)_1= \sum_{\alpha\in
L_1} t^\alpha \left(\widetilde {M^{n,\alpha}}\right)_1=\Phi^{L,A}_m
\left(\sum_{\alpha \in L_1} (t^\alpha\otimes M^{n,\alpha}_1)\right).
$$
So, $\Phi^{L,A}_m$ is surjective. Moreover, since $L$ is faithfully
flat over $k$, $\Leap_k(A)=\Leap_L(A_L)$ by Lemma \ref{Saltos y
cambio de base}.
\end{proof}


Let $L\supseteq k$ a pure transcendental field extension. Then, we
can express $L=T^{-1}L'$ where $L'=k[t_i\ |\mbox{ } i \in \I]$ and
$T=L'\setminus \{0\}$. Hence, for any finitely generated $k$-algebra
$A$, we have that
$$
L\otimes_k\IDer_k(A;m)\cong T^{-1}L'\otimes_{L'}
L'\otimes_k\IDer_k(A;m)\cong T^{-1}L' \otimes_{L'}
\IDer_{L'}(A_{L'};m).
$$
Now, let us recall the following proposition:

\begin{prop}{\rm \cite[{\bf Corollary 2.3.5}]{Na2}}\label{Localizacion e
integrables} Assume that $B$ is a finitely presented $C$-algebra,
where $C$ is a commutative ring, and let $T\subseteq B$ be a
multiplicative set. Then, for any integer $m\geq 1$, the canonical
map
$$
T^{-1}\IDer_C(B;m)\rightarrow  \IDer_C(T^{-1}B;m)
$$
is an isomorphism of $(T^{-1}B)$-modules.
\end{prop}

Hence, if $A$ is finitely presented $k$-algebra,
$T^{-1}L'\otimes_{L'} \IDer_{L'}(A_{L'};m)\cong \IDer_{L'}
(T^{-1}L'\otimes_{L'} A_{L'};m)=\IDer_{L'} (A_L;m)$. Moreover, it is
easy to prove that if $T\subseteq L'$, then any Hasse-Schmidt
derivation over $L'$ is $T^{-1}L'$-linear, so
$\HS_{L'}(A_L;m)=\HS_{T^{-1}L'}(A_L;m)$. Therefore,
$$
L\otimes_k\IDer_k(A;m)\cong
T^{-1}L'\otimes_{L'}\IDer_{L'}(A_{L'};m)\cong\IDer_{L}(A_L;m)
$$
and  we have proved the following corollary:
\begin{cor}\label{Trascendentes}
Let $k$ be a field and $L$ a pure transcendental field extension of $k$.
Assume that $A$ is a finitely presented $k$-algebra. Then, $
\Phi^{L,A}_m:L\otimes_k \IDer_k(A;m)\rightarrow \IDer_L(A_L;m) $ is
an isomorphism of $A_L$-modules for all $m\in \N$. Moreover, $\Leap_k(A)=\Leap_L(A_L)$.
\end{cor}

\subsubsection{Separable extensions}

Let us consider a field $k$ of characteristic $p>0$ and $L$ a
$k$-algebra containing $k$. Recall that $L$ is separable over $k$ if
$L_K:=K\otimes_k L$ is reduced for every possible extension $K$ of
$k$. In this section, we prove that $\Phi_m^{L,A}$ is bijective when
$L$ is a separable algebra over a field $k$ and $A$ a finitely
generated $k$-algebra.

\begin{hip}\label{Hip para bases} Let $k$ be a ring of characteristic $p>0$
and $k\to L$ a free ring extension. Then, we assume that the following conditions hold.
\begin{itemize}
\item[1.] For every k-linearly independent subset
$\{a_i,\ i\in \mathcal I\}$ of $L$, the subset
$\left\{a_i^p,\mbox{ } i\in \I\right\}$ of $L$ continues to be
$k$-linearly independent.
\item[2.] For every $k$-basis $\{a_i,\mbox{ }i\in \mathcal I\}$ of $L$  and every
$k$-linearly independent set $\{b_1,\ldots,b_s\}$ of $L$, there
exists $\mathcal L\subseteq \mathcal I$ such that
$\{b_1,\ldots,b_s\}\cup \{a_i,\mbox{ }i\in \mathcal L\}$ is a
$k$-basis of $L$.
\end{itemize}
\end{hip}

\begin{nota}\label{Separable cumple Hip}
If $k$ is a field, then the second condition always holds and the
first one is equivalent to $L$ being a separable $k$-algebra (see
\cite[\S 15.4. Th. 2]{Bo}). Then, if $L$ is a separable $k$-algebra,
$L$ satisfies Hypothesis \ref{Hip para bases}. Unfortunately, we do
not know another type of extension that satisfies Hypothesis
\ref{Hip para bases}.
\end{nota}

From now on, we put $R=k[x_1,\ldots,x_d]$.

\begin{hip}\label{Hip para sobrey}
Let $l\geq 1$ be an integer. We say that $I\subseteq R$ satisfies $S_l$ if
$\Phi_{\text{ind},m}^{L,R}:L\otimes_k \IDer_k(\log I;m)\to
\IDer_L(\log I^e;m)$ is surjective of all $m<p^l$.
\end{hip}

Note that if $k\to L$ is a flat ring extension where $k$ is a ring of characteristic $p>0$, $S_1$ is satisfied
for all $I\subseteq R$ thanks to $\Phi_{\text{ind}}^{L,R}:L\otimes_k
\Der_k(\log I)\to \Der_k(\log I^e)$ is bijective and leaps only
occur at powers of $p$.

\begin{lem}\label{Derivaciones pl-1-log}
Let $l\geq 1$ be an integer and $k$ a ring of characteristic $p>0$.
Assume that $k\to L$ is a free ring extension and $I\subseteq R$
satisfies $S_l$. Let us consider a
$\left(p^l-1\right)-I$-logarithmic HS-derivation $D\in
\HS_L\left(R_L;p^l\right)$. Then, for each $k$-basis $\{a_i,\mbox{
}i\in \I\}$ of $L$, there exist a finite subset $\I_0\subseteq \I$
and a $\left(p^l-1\right)-I$-logarithmic HS-derivation $N^i\in
\HS_k\left(R;p^l\right)$ for each $i\in \I_0$ such that if
$$
E=\circ_{i\in \I_0} \left(a_i\bullet \widetilde{N^i}\right)
$$
(where we choose any order of composition) there exist a
$\left(p^{l-1}-1\right)-I^e$-logarithmic HS-derivation $T\in
\HS_L\left(R_L;p^{l-1}\right)$ and an $I^e$-logarithmic
HS-derivation $F\in \HS_L\left(\log I^e;p^l\right)$ with $\ell(F)>1$
such that
$$
D=E\circ T[p]\circ F.
$$
\end{lem}

\begin{proof}
Since $\Phi_{\text{ind},p^l-1}^{L,R}:L\otimes_k \IDer_k(\log
I;p^l-1)\to \IDer_L(\log I^e;p^l-1)$ is surjective and $D_1\in
\IDer_L(\log I^e;p^l-1)$, there exist a subset $\I_0\subset\I$ and
$\delta_i\in \IDer_k(\log I;p^l-1)$ for each $i\in \I_0$ such that
$$
\Phi^{L,R}_{\text{ind},p^l-1}\left(\sum_{i\in \I_0} a_i\otimes
\delta_i\right)=\sum_{i\in \I_0} a_i\widetilde\delta_i=D_1.
$$
Let us consider a $(p^l-1)-I$-logarithmic integral $N^i\in
\HS_k(R;p^l)$ of $\delta_i$ for all $i\in \I_0$. Then, $
E:=\circ_{i\in \I_0} \left(a_i\bullet \widetilde{N^i}\right) $ is a
$p^l$-integral of $D_1$ (note that the order of the composition is
not important, $E$ is always an integral of $D_1$). Since $N^i$ is
$(p^l-1)-I$-logarithmic for all $i\in \I_0$, we have that
$\widetilde{N^i}$ is $(p^l-1)-I^e$-logarithmic (see Lemma
\ref{Propiedades Extension} (3)). Hence, by Lemma \ref{lema-resumen}
(b) and (d), $E^\ast$ is a $(p^l-1)-I^e$-logarithmic integral of
$-D_1$. Therefore, $E^\ast\circ D\in \HS_L(R_L;p^l)$ is a
$(p^l-1)-I^e$-logarithmic HS-derivation such that $\ell(E^\ast \circ
D)>1$. So, we can apply Corollary \ref{Descomposicion de HS en
partes no log y log} to this HS-derivation. Then, there exist a
$(p^{l-1}-1)-I^e$-logarithmic HS-derivation $T\in
\HS_L(R_L;p^{l-1})$ and $F\in \HS_L(\log I^e;p^l)$ with $\ell(F)>1$
such that
$$
E^\ast\circ D=T[p]\circ F \Rightarrow D=E\circ T[p]\circ F
$$
and the result is proved.
\end{proof}

\begin{teo}\label{Expresion para D en hip para base}
Let $l\geq 1$ be an integer and assume that $k\to L$ satisfies
Hypothesis \ref{Hip para bases} and the ideal $I\subseteq R$
satisfies $S_l$. Let us consider a
$\left(p^l-1\right)-I^e$-logarithmic HS-derivation
$D\in\HS_L(R_L;p^l)$. Then, for every $k$-basis $\{a_i,\mbox{ }i\in
\I\}$ of $L$, there exist, for all $j=0,\ldots, l$,
\begin{itemize}
\item a finite subset $\I_j$ of $\I$ and
\item a $\left(p^{l-j}-1\right)-I$-logarithmic HS-derivation $N^{j,n,i,j-n}\in
\HS_k\left(R;p^{l-j}\right)$ for each $i\in \I_{j-n}$, $0\leq n\leq
j$
\end{itemize}
such that for all $j=0,\ldots,l$
$$
\bigcup_{m=0}^j \left\{ a_i^{p^{j-m}},\mbox{ }i\in \I_m\right\}
\text{ is a $k$-linearly independent set of } L
$$
and, if we take
$$
E^j=\circ_{i\in \mathcal I_0} \left(a_i^{p^j}\bullet
\widetilde{N^{j,j,i,0}}\right)\circ \circ_{i\in \mathcal I_1}
\left(a_i^{p^{j-1}}\bullet \widetilde{N^{j,j-1,i,1}}\right)\circ
\cdots\circ \circ_{i\in \mathcal I_j} \left(a_i\bullet
\widetilde{N^{j,0,i,j}}\right)
$$
for all $j=0,\ldots, l$ then, there exists $F\in \HS_L(\log
I^e;p^l)$ with $\ell(F)>1$ such that
$$
D=E^0\circ E^1[p]\circ \cdots \circ E^l\left[p^l\right] \circ F.
$$
\end{teo}

\begin{proof}
By Lemma \ref{Derivaciones pl-1-log}, there exist a finite subset
$\I_0\subseteq \I$ and  a $(p^l-1)-I$-logarithmic HS-derivation
$N^{0,0,i,0}\in \HS_k(R;p^l)$ for each $i\in \I_0$ such that, if we
take $E^0=\circ_{i\in \I_0} \left(a_i\bullet N^{0,0,i,0}\right)$,
there exist a $(p^{l-1}-1)-I^e$-logarithmic HS-derivation $T^1\in
\HS_L(R_L;p^{l-1})$ and $F\in \HS_L(\log I^e;p^l)$ with $\ell(F)>1$
such that
$$
D=E^0\circ T^1[p]\circ F.
$$
Observe that the set $\mathcal C_0:=\{a_i,\mbox{ } i\in \I_0\}$ of
$L$ is $k$-linearly independent so, by Hypothesis \ref{Hip para
bases}, we have that the set $\mathcal C_0^p:=\{a_i^p,\mbox{ }i\in
\I_0\}$ of $L$ is also $k$-linearly independent and from the point 2
in Hypothesis \ref{Hip para bases} (taking $\{a_i,\mbox{ } i\in
\I\}$ as $k$-basis) we obtain a subset $\mathcal L_1\subseteq \I$
such that $\mathcal B_1=\mathcal C_0^p \cup \{a_i,\mbox{ }i\in
\mathcal L_1\}$ is a $k$-basis of $L$.  Note that if $l\neq 1$, we
can apply the previous lemma to $T^1$ using the $k$-basis $\mathcal
B_1$ of $L$.

\medskip

\begin{La induccion}
Let us suppose that doing this process recursively we obtain that,
for some integer $j$ such that $0\leq j\leq l$, there exist for all
$s=0,\ldots, j-1$,
\begin{itemize}
\item a finite subset $\I_s$ of $\I$,
\item a $(p^{l-s}-1)-I$-logarithmic HS-derivation $N^{s,n,i,s-n}\in
\HS_k(R;p^{l-s})$ for all $i\in \I_{s-n}$ and  $0\leq n\leq s$
\end{itemize}
such that for all $s=0,\ldots,j-1$,
$$
\mathcal C_s=\bigcup_{m=0}^s \left\{a_i^{p^{s-m}},\mbox{ }
i\in\I_m\right\} \text{ is $k$-linearly independent set of $L$}
$$
and if we take
$$
E^s=\circ_{i\in \mathcal I_0} \left(a_i^{p^s}\bullet
\widetilde{N^{s,s,i,0}}\right)\circ \circ_{i\in \mathcal I_1}
\left(a_i^{p^{s-1}}\bullet \widetilde{N^{s,s-1,i,1}}\right)\circ
\cdots\circ \circ_{i\in \mathcal I_s} \left(a_i\bullet
\widetilde{N^{s,0,i,s}}\right)
$$
for all $s=0,\ldots,j-1$ then, there exist
\begin{itemize}
\item $F\in \HS_L(\log I^e;p^l)$ with $\ell(F)>1$ and
\item  a $(p^{l-j}-1)-I^e$-logarithmic HS-derivation $T^j\in
\HS_L(R_L;p^{l-j})$
\end{itemize}
such that
\begin{equation}\label{Separ.1}
D=E^0\circ E^1[p]\circ \cdots \circ E^{j-1}\left[p^{j-1}\right]\circ
T^j\left[p^j\right] \circ F.
\end{equation}
\end{La induccion}

Observe that since $\mathcal C_{j-1}$ is $k$-linearly independent,
then $\mathcal C_{j-1}^p=\bigcup_{m=0}^{j-1}\left\{
a_i^{p^{j-m}},\mbox{ }i\in \I_m\right\}$ is also a $k$-linearly
independent finite set of $L$. So, there exists a subset $\mathcal
L_j\subseteq \I$ such that $\mathcal B_j:=\mathcal C_{j-1}^p \cup
\{a_i,\mbox{ }i\in \mathcal L_j\}$ is a $k$-basis of $L$ (see
Hypothesis \ref{Hip para bases}).

\medskip
 Let us suppose that $j\neq l$, i.e. $l-j\geq 1$. Then, we can
apply Lemma \ref{Derivaciones pl-1-log} to $T^j$ using the $k$-basis
$\mathcal B_j$ of $L$. Hence, there exists a finite subset $\mathcal
I'_m$ of $\I_m$ for all $m=0,\ldots,j-1$, a finite set $\I'_j$ of
$\mathcal L_j$ and a $(p^{l-j}-1)-I$-logarithmic HS-derivation
$N^{j,n,i,j-n}\in \HS_k(R;p^{l-j})$ for each $0\leq n\leq j$ and
$i\in \I'_{j-n}$ such that, if we take
$$
E^j=\circ_{i\in \mathcal I'_0} \left(a_i^{p^j}\bullet
\widetilde{N^{j,j,i,0}}\right)\circ \circ_{i\in \mathcal I'_1}
\left(a_i^{p^{j-1}}\bullet \widetilde{N^{j,j-1,i,1}}\right)\circ
\cdots\circ \circ_{i\in \mathcal I_j} \left(a_i\bullet
\widetilde{N^{j,0,i,j}}\right)
$$
then, there exist $F'\in \HS_L\left(\log I^e;p^{l-j}\right)$ with
$\ell(F')>1$ and a $(p^{l-(j+1)}-1)-I^e$-logarithmic HS-derivation
$T^{j+1}\in \HS_L\left(R_L;p^{l-(j+1)}\right)$ such that
$$
T^j=E^j\circ T^{j+1}[p]\circ F'.
$$
Note that we can take $\I_m'=\I_m$ for all $0\leq n\leq j-1$ (it is enough to take
$N^{j,n,i,j-n}=\II$ for all $i\in \I_m\setminus \I'_m$) and let us rewrite $\I_j=\I_j'$. Moreover,
the subset $C_j=\bigcup_{m=0}^j \left\{a_i^{p^{j-m}},\mbox{ }i\in
\I_m\right\}$ of $L$ is $k$-linearly independent and, if we replace
$T^j$ in (\ref{Separ.1}), we obtain that
$$
D=E^0\circ \cdots \circ E^{j-1}\left[p^{j-1}\right]\circ
E^j[p^j]\circ T^{j+1}[p^{j+1}]\circ F'[p^j]\circ F.
$$
Observe that $F[p^j]\in \HS_L(\log I^e;p^l)$ so, $F:=F'[p^j]\circ
F\in \HS_L(\log I^e;p^l)$ with $\ell(F)>1$. Therefore, we have the
same condition that {\it Assumption} for $j+1$. So that, we can
apply this process until $j=l$.

\medskip

Let us suppose that $j=l$ in {\it Assumption}. Then, $T^l\in
\HS_L(R_L;1)\equiv \Der_L(R_L)$ and, by the proof of Proposition
\ref{prop:free-extension-der} with the $k$-basis $\mathcal
B_j=\mathcal B_l$, there exists a finite subset $\I_l\subseteq
\mathcal L_l\subseteq \I$ such that
$$
T^l=\circ_{i\in \I_0}
\left(a_i^{p^j}\bullet\widetilde{N^{l,l,i,0}}\right)\circ
\circ_{i\in \I_1} \left(a_i^{p^{j-1}}\bullet
\widetilde{N^{l,l-1,i,1}}\right)\circ\cdots \circ \left(\circ_{i\in
\I_l} a_i \bullet\widetilde{N^{l,0,i,l}}\right)
$$
where $N^{l,n,i,l-n}\in \HS_k(R;1)$ for each $i\in \I_{l-n}$ and
$0\leq n\leq l$. It is obvious that $\bigcup_{m=0}^l
\left\{a_i^{p^{j-m}},\mbox{ }i\in \I_{j-m}\right\}$ is a
$k$-linearly independent set of $L$ and since $D=E_0\circ
E_1[p]\circ \cdots \circ E^{l-1}[p^{l-1}]\circ T^l[p^l]\circ F$, we
have the result.
\end{proof}

\begin{teo}\label{Phi Sobrey si cumple hip para base}
Let $k\to L$ be a ring extension satisfying Hypothesis \ref{Hip para
bases} and $A$ a commutative finitely generated $k$-algebra. Then,
$\Phi_m^{L,A}:L\otimes_k \IDer_k(A;m)\to \IDer_L(A_L;m)$ is
an isomorphism of $A_L$-modules for all $m\in \N$. Moreover, $\Leap_k(A)=\Leap_L(A_L)$.
\end{teo}

\begin{proof}  If $\Phi_m^{L,A}$ is bijective, since $L$ is faithfully flat over $k$, we have that $\Leap_k(A)=\Leap_L(A_L)$ by Lemma \ref{Saltos y cambio de base}. Moveover, by Lemma
\ref{Inyectividad/Sobreyectividad} 1., $\Phi_m^{L,A}$ is injective for all $m\in \N$. So, we only need to prove that $\Phi^{L,A}_m$ is surjective.

Recall that we consider $A=R/I$ where
$R=k[x_1,\ldots,x_d]$ is a polynomial ring in a finite number of
variable and $I\subseteq R$ an ideal. By Lemma
\ref{Inyectividad/Sobreyectividad},  $\Phi_m^{L,A}$ is surjective if and only if
$\Phi^{L,R}_{\text{ind},m}:L\otimes_k \IDer_k(\log I;m)\to
\IDer_L(\log I^e;m)$ is surjective. So, we will prove that
$\Phi_{\text{ind},m}^{L,R}$ is surjective for all $m\in \N$.
Moreover, since leaps only occur at powers of $p$, it is enough to
see that $\Phi_{\text{ind},m}^{L,R}$ is surjective when $m=p^l$ for
$l\geq 0$. We proceed by induction on $l\geq 0$.

\medskip
If $l=0$, Proposition \ref{prop:free-extension-der} gives us the
result in this case. Now, let us assume that
$\Phi_{\text{ind},m}^{L,R}$ is surjective for all $m<p^l$ with
$l\geq 1$, i.e. $I$ satisfies $S_l$, and we prove the theorem for
$\Phi_{\text{ind},p^l}^{L,R}$ with $l\geq 1$.

Let $\delta\in \IDer_L(\log I^e,p^l)$ be an $L$-derivation of $R_L$,
then there exists $D\in \HS_k(\log I^e;p^l)$ an integral of
$\delta$. In particular, $D$ is $(p^l-1)-I^e$-logarithmic and we can
apply Theorem \ref{Expresion para D en hip para base} to $D$. Let us
consider a $k$-basis $\{a_i,\mbox{ }i\in \I\}$ of $L$. Then, for all
$j=0,\ldots, l$, there exist
\begin{itemize}
\item a finite subset $\I_j$ of $\I$ and
\item a $(p^{l-j}-1)-I$-logarithmic HS-derivation $N^{j,n,i,j-n}\in
\HS_k(R;p^{l-j})$ for each $i\in \I_{j-n}$ and $0\leq n\leq j$
\end{itemize}
such that, for all $j=0,\ldots, l$ the subset
$$
\bigcup_{m=0}^j \left\{a_i^{p^{j-m}},\mbox{ }i\in \I_m\right\}
\mbox{ of $L$ is $k$-linearly independent}
$$
and, if we take
$$
E^j=\left(\circ_{i\in \I_0} a_i^{p^j}\bullet
\widetilde{N^{j,j,i,0}}\right)\circ \cdots \circ \left(\circ_{i\in
\I_j} a_i\bullet \widetilde{N^{j,0,i,j}}\right)
$$
for all $j=0,\ldots, l$, there exists $F\in \HS_L(\log I^e;p^l)$
with $\ell(F)>1$ such that
$$
D=E^0\circ E^1[p]\circ \cdots \circ E^{l}\left[p^l\right]\circ F.
$$
For each $j=0,\ldots,l$, $N^{j,n,i,j-n}$ is
$\left(p^{l-j}-1\right)-I$-logarithmic for all $0\leq n\leq j$ and
$i\in \I_{j-n}$. So, $\widetilde{N^{j,n,i,j-n}}$ is
$\left(p^{l-j}-1\right)-I^e$-logarithmic for all $0\leq n\leq j$ and
$i\in \I_{j-n}$ (see Lemma \ref{Propiedades Extension} (3)).
Therefore, by Lemma \ref{lema-resumen} (d), $E^j\in
\HS_L(R_L;p^{l-j})$ is $\left(p^{l-j}-1\right)-I^e$-logarithmic and
$$
E^j_{p^{l-j}}=\sum_{i\in \I_0} \left(a_i^{p^j}\right)^{p^{l-j}}
\widetilde{N^{j,j,i,0}_{p^{l-j}}}+ \cdots +\sum_{i\in \I_j}
a_i^{p^{l-j}} \widetilde{N^{j,0,i,j}_{p^{l-j}}}+\text{\it some
$I^e$-diff. op.}
$$
Hence, from Lemma \ref{lema-resumen} (a), $E^j\left[p^j\right]\in
\HS_L\left(R_L;p^l\right)$ is $\left(p^l-1\right)-I^e$-logarithmic
for all $j$ and
$$
E^j[p^j]_{p^l}=E_{p^{l-j}}^j=\sum_{k=0}^j \sum_{i\in \I_k}
a_i^{p^{l-k}} \widetilde{N_{p^{l-j}}^{j,j-k,i,k}}+\mbox{\it some
$I^e$-diff. op.}
$$
So, by Lemma \ref{lema-resumen} (d),
$$
D_{p^l}=\sum_{j=0}^l E^j[p^j]_{p^l}+\mbox{\it some $I^e$-diff.
op.}=\sum_{j=0}^l\sum_{k=0}^j \sum_{i\in \I_k} a_i^{p^{l-k}}
\widetilde{N_{p^{l-j}}^{j,j-k,i,k}}+\mbox{\it some $I^e$-diff. op.}
$$
Since $D_{p^l}$ is an $I^e$-differential operator,
$$
\sum_{j=0}^l\sum_{k=0}^j \sum_{i\in \I_k} a_i^{p^{l-k}}
\widetilde{N_{p^{l-j}}^{j,j-k,i,k}}=\sum_{i\in \I_0}a_i^{p^l}
\left(\sum_{j=0}^l
\widetilde{N_{p^{l-j}}^{j,j,i,0}}\right)+\sum_{i\in \I_1}
a_i^{p^{l-1}}\left(\sum_{j=1}^l
\widetilde{N_{p^{l-j}}^{j,j-1,i,1}}\right)+\cdots+\sum_{i\in \I_l}
a_i\widetilde{N_1^{l,0,i,l}}
$$
is an $I^e$-differential operator.

\smallskip

Since $\mathcal C:=\bigcup_{k=0}^l \left\{a_i^{p^{l-k}},\mbox{ }i\in
\I_k\right\}$ is a $k$-linearly independent finite set of $L$ and
$\{a_i,\mbox{ }i\in \I\}$ is a $k$-basis of $L$, by Hypothesis
\ref{Hip para bases}, there exists $\mathcal L\subseteq \mathcal I$
such that $\mathcal C\cup \{a_i,\mbox{ }i\in \mathcal L\}$ is a
$k$-basis of $L$. Hence, we can deduce, in the same way that in the
proof of Proposition \ref{prop:free-extension-der}, that
$$
\sum_{j=0}^l N_{p^{l-j}}^{j,j,i,0} \mbox{ is an }
I\mbox{-differential operator for all }
i\in \I_0
$$
(note that
$\widetilde{N_{p^{l-j}}^{j,j,i,0}}_{|R}=N_{p^{l-j}}^{j,j,i,0}$).

For all $i\in \I_0$, let us consider $D^i=N^{0,0,i,0}\circ
N^{1,1,i,0}[p]\circ \cdots\circ N^{l,l,i,0}\left[p^l\right]\in
\HS_k(R;p^l)$ an integral of $N^{0,0,i,0}_1$. Since $N^{j,j,i,0}\in
\HS_k(R;p^{l-j})$ is $(p^{l-j}-1)-I$-logarithmic for all $j=0,\ldots
,l$, $N^{j,j,i,0}[p^j]\in \HS_k(R;p^l)$ is $(p^l-1)-I$-logarithmic
and by Lemma \ref{lema-resumen} (d), $D^i\in \HS_k(R;p^l)$ is
$(p^l-1)-I$-logarithmic and
$$
D^i_{p^l}=\sum_{j=0}^l N^{j,j,i,0}_{p^{l-j}}+\text{\it some
$I$-differential operator}
$$
So, $D^i\in \HS_k\left(\log I;p^l\right)$ and we can deduce that
$N^{0,0,i,0}_1\in\IDer_k\left(\log I;p^l\right)$. On the other hand,
we recall that
$$
D=E^0\circ E^1[p]\circ\cdots\circ E^l\left[p^l\right] \circ F
$$
where $\ell(F)>1$. Then, $D_1=E^0_1$ and, since $E^0=\circ_{i\in
\I_0}\left( a_i\bullet \widetilde{N^{0,0,i,0}}\right)$, we have that
$$
D_1=\sum_{i\in \I_0} a_i
\widetilde{N_1^{0,0,i,0}}=\Phi^{L,R}_{\text{ind},p^l}
\left(\sum_{i\in \I_0} \left(a_i\otimes N_1^{0,0,i,0}\right)\right)
$$
Therefore, $\Phi^{L,R}_{\text{ind},m}$ is bijective.
\end{proof}

\begin{nota} If we change the condition 2. in
Hypothesis \ref{Hip para bases} for
\begin{itemize}
\item[2'.] There exists a $k$-basis $\{a_i,\mbox{ }i\in \I\}$  of $L$ such that
$\left\{a^{p^r},\mbox{ }i\in \I\right\}\subseteq \{a_i,\mbox{ }i\in \I\}$ for all
$r\geq 1$.
\end{itemize}
then, Theorems \ref{Expresion para D en hip para base} and \ref{Phi
Sobrey si cumple hip para base} are true for that basis. For
example, if we take $L=k[t_i\ |\ i\in \I]$, we can apply these
theorems and we obtain that $\Phi_m^{L,A}$ is an isomorphism.
\end{nota}

\begin{cor}\label{Separables}
Let $k$ be a field of characteristic $p>0$, $k\to L$ a separable
extension and $A$ a commutative finitely generated $k$-algebra.
Then, $\Phi^{L,A}_m:L\otimes_k \IDer_k(\log I;m)\rightarrow
\IDer_k(\log I^e;m)$ is an isomorphism of $A_L$-modules for all $m\geq 1$. Moreover, $\Leap_k(A)=\Leap_L(A_L)$.
\end{cor}

{\bf Acknowledgment.} The author thanks Professor Luis Narv\'aez
Macarro for his careful reading of this paper with numerous useful
comments.

\end{document}